\documentclass[reqno,11pt]{amsart}

\usepackage{amsmath,amssymb,amsthm,graphicx,url,mathrsfs}
\usepackage{wrapfig}
\usepackage{enumitem}
\usepackage{mathtools}
\usepackage{xparse}
\usepackage[usenames,dvipsnames]{xcolor}
\usepackage[colorlinks=true,linkcolor=Red,citecolor=Green]{hyperref}
\usepackage[super]{nth}
\usepackage[open, openlevel=2, depth=3, atend]{bookmark}
\hypersetup{pdfstartview=XYZ}
\usepackage[font=footnotesize]{caption}
\usepackage{a4wide}
\usepackage{color}

\usepackage[utf8]{inputenc}

\numberwithin{equation}{section}

\newcommand{\eps}{\epsilon}

\newcommand{\mc}{\mathcal}

\newcommand{\pl}{\partial}
\newcommand{\x}{\times}

\newcommand{\bbar}{\overline}

\newcommand{\cjd}{\rangle}
\newcommand{\cjg}{\langle}

\newcommand{\R}{\mathbb{R}}

\newcommand{\supp}{\operatorname{supp}}

\newcommand{\Z}{\mathbb{Z}}

\renewcommand{\Re}{\operatorname{Re}}
\renewcommand{\Im}{\operatorname{Im}}

\newcommand{\jap}[1]{\left\langle #1 \right\rangle}
\newcommand{\p}[1]{\left(#1\right)}
\newcommand{\va}[1]{\left|#1\right|}
\newcommand{\n}[1]{\left\| #1 \right\|}
\newcommand{\set}[1]{\left\{#1\right\}}
\DeclareMathOperator{\WF}{WF}

\newtheorem{proposition}{Proposition}[section]
\newtheorem{lemma}[proposition]{Lemma}

\newtheorem{definition}[proposition]{Definition}
\newtheorem{corol}[proposition]{Corollary}
\newtheorem{theorem}{Theorem}
\theoremstyle{remark}
\newtheorem{remark}{Remark}

\newtheorem*{conjecture}{Conjecture}

\author{Yannick Guedes Bonthonneau}

\address{Université Paris-Nord, CNRS, LAGA, Villetaneuse, France.}

\email{bonthonneau@math.univ-paris13.fr}

\author{Colin Guillarmou}

\address{Université Paris-Saclay, CNRS, Laboratoire de mathématiques d’Orsay, 91405, Orsay, France.}

\email{colin.guillarmou@universite-paris-saclay.fr}

\author{Malo Jezequel}

\address{Department of Mathematics, Massachusetts Institute of Technology, Cambridge, MA 02139.}

\email{mpjez@mit.edu}

\title[Scattering rigidity]{Scattering rigidity for analytic metrics}
\begin{document}

\begin{abstract}
For analytic negatively curved Riemannian manifolds with analytic strictly convex boundary, 
we show that the scattering map for the geodesic flow determines the manifold up to isometry. In particular, one recovers both the topology and the metric. 
More generally 
our result holds in the analytic category under the no conjugate point and hyperbolic trapped set assumptions.
\end{abstract}

\maketitle

\section{Introduction}

On a compact connected Riemannian manifold $(M,g)$ with boundary $\pl M$, denote by 
\[\mc{G}(M,g):= \{\gamma: [0,1]\to M \,| \, \gamma \, \textrm{ is  a geodesic for }g, (\gamma(0),\gamma(1))\in (\pl M)^2\}\]
the set of geodesics with endpoints in the boundary.
The \emph{scattering data} of $(M,g)$ is the set 
\begin{align*}
 {\rm Sc}(M,g):=\Big\{ \Big(\gamma(0),\frac{\dot{\gamma}(0)}{|\dot{\gamma}(0)|_g},\gamma(1),\frac{\dot{\gamma}(1)}{|\dot{\gamma}(1)|_g}\Big) \in (T_{\pl M}M)^2  \,\Big| \, \gamma\in \mc{G}(M,g)\Big\}, 
 \end{align*} 
of endpoints, together with the normalized tangent vector at the endpoints, of all geodesics in $\mc{G}(M,g)$. This is the graph of a map $S_g$, called scattering map, defined on a subset of the set of incoming tangent vectors at $\pl M$ and mapping to a subset of the set of outgoing tangent vectors at $\pl M$; see Figure \ref{scatteringmap}. \begin{center}
\begin{figure}
\includegraphics[scale=0.75]{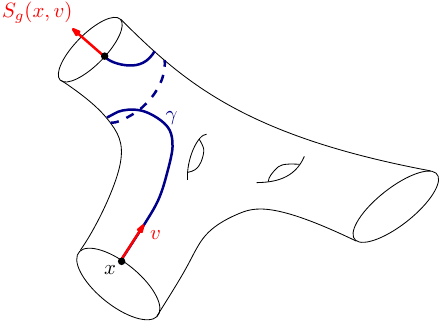}
\caption{The point $(x,v;S_g(x,v))\in \mc{Sc}(M,g)$ corresponds to the geodesic $\gamma\in \mc{G}(M,g)$. The map $(x,v)\mapsto S_g(x,v)$ is called the scattering map.}
\label{scatteringmap}
\end{figure}
\end{center}
A natural geometric inverse problem consists in determining the manifold $M$ and the metric $g$ from the scattering data ${\rm Sc}(M,g)$. 
This problem is called the \emph{scattering rigidity} problem. A manifold $(M,g)$ is called scattering rigid if it is the only Riemannian 
manifold, up to isometry being the identity on the boundary, with boundary equal (or isometric) to $(\pl M,g|_{\pl M})$ and with scattering data given by ${\rm Sc}(M,g)$. 
Here, we notice that the tangent space $T_{\pl M}M$ of $M$ at $\pl M$ can be identified with $T\pl M\oplus \R\nu$ where $\nu$ is the inward pointing unit normal 
vector field to $\pl M$ for $g$, so that the scattering data ${\rm Sc}(M,g)$ and ${\rm Sc}(M',g')$ can be compared if the boundary are isometric as Riemannian manifold $(\pl M,g|_{\pl M})\simeq (\pl M',g'|_{\pl M'})$ (by identifying the inward pointing normal $\nu$ and $\nu'$ for $g$ and $g'$). 
To simplify notations, we always write $(\pl M,g|_{T\pl M})=(\pl M',g'|_{T\pl M'})$ to mean that the boundary are isometric.

As far as we know, there are only very few cases of manifolds that are known to be scattering rigid:
\begin{itemize} 
\item the product $B_{\R^n}(0,1)\times \mathbb{S}^1$ equipped with the flat product metric --- Croke \cite{Croke14},
\item Simple Riemannian surfaces --- Wen \cite{Wen15}.
\end{itemize}
We recall that a Riemannian manifold $(M,g)$ is called \emph{simple} if $M$ is topologically a ball, the boundary $\pl M$ is strictly convex (i.e. its second fundamental form is positive) and 
$g$ has no conjugate points (see e.g. \cite{PSUbook}).  

There are however more results of scattering rigidity within certain classes of Riemannian manifolds. To be precise, we say that a manifold $(M,g)$ is scattering rigid within the class of Riemannian manifold with a given property $\mc{P}$ if $(M,g)$ has the property $\mc{P}$ and $(M,g)$ is the only (up to isometry equal to the identity on the boundary) Riemannian manifold with the property $\mc{P}$ that has the scattering data ${\rm Sc}(M,g)$.
The scattering rigidity within the class of simple Riemannian manifolds 
is equivalent to the so-called \emph{boundary rigidity problem} posed by Michel \cite{Michel-81}:

\emph{For simple Riemannian manifolds, does the Riemannian distance between all pairs of boundary points determine the metric up to isometry equal to the identity on the boundary?}

In particular, scattering rigidity within the class of simple Riemannian surfaces with negative curvature (resp. non-positive curvature) is a consequence of the proof of boundary rigidity by Otal \cite{Otal:1990ko} (resp. Croke \cite{Croke-90}), while the resolution of the problem within the class of general simple Riemannian surfaces is a result of Pestov--Uhlmann \cite{Pestov-Uhlmann}. In dimension $n\geq 3$, the scattering rigidity among simple non-positively curved Riemannian metrics follows from the recent works of Stefanov--Uhlmann--Vasy \cite{Stefanov:2017lt}. We refer in general to the recent book by Paternain--Salo--Uhlmann \cite{PSUbook} for an overview of the subject on this problem.\\

For what concerns Riemannian manifolds with boundary that are not simple, in particular those which are not simply connected and have geodesics of infinite length, the only results in the direction of scattering rigidity are the special example of \cite{Croke14} mentionned above, 
and the result of the second author \cite{Guillarmou17} stating that two negatively curved surfaces with strictly convex boundary and the same scattering data are conformally equivalent.\\
 
In this work, we prove that scattering rigidity holds within the class of real analytic negatively curved manifolds:

\begin{theorem}\label{Th1}
Let $(M_1,g_1)$ and $(M_2,g_2)$ be two real analytic negatively curved compact connected Riemannian manifolds (of any dimension) with non-empty analytic strictly convex boundary. Assume that $(\pl M_1,h_1)=(\pl M_2,h_2)$ where $h_i:=g_i|_{T\pl M_i}$ is the metric on the boundary, and that their scattering data
 ${\rm Sc}(M_1,g_1)={\rm Sc}(M_2,g_2)$ agree. Then there exists an analytic diffeomorphism $\psi:M_1\to M_2$ such that $\psi|_{\pl M_1}={\rm Id}$ and 
$\psi^*g_2=g_1$.
\end{theorem}
In particular for analytic negatively curved surfaces, the scattering data determines both the topology and the geometry, and not only the conformal class as in \cite{Guillarmou17}.

We note that the class of manifolds we consider have in general trapped geodesics, that is infinite length geodesics 
that do not intersect the boundary. In fact, we prove a result for a more general class of manifolds, as we now explain. 
Let $(M,g)$ be a compact Riemannian manifold with strictly convex boundary. Let $SM:=\{(x,v)\in TM\, | \, |v|_{g_x}=1\}$ be the unit tangent bundle 
and $\varphi_t^g:SM\to SM$ be the geodesic flow at time $t$, with $X_g$ its generating vector field. The trapped set is the closed set defined by
\[ 
\mc{K}^g:=\{ (x,v)\in SM\,|\, \forall t\in \R,\, \varphi^g_t(x,v)\in \mathring{M}\}.
\]
This is an invariant set under $\varphi_t^g$.  We will assume that the trapped set $\mc{K}^g$ is \emph{hyperbolic} for the geodesic flow, i.e there exists a continuous flow-invariant splitting 
of $T(SM)$ over $\mc{K}^g$:
\[ 
\forall z \in \mc{K}^g, \quad T_{z}(SM)=\R X_g(z)\oplus E_s(z)\oplus E_u(z)
\]
and constants $C>0,\nu>0$ such that $\|\mathrm{d}\varphi^g_t|_{E_s}\|\leq Ce^{-\nu t}$ and $\|\mathrm{d}\varphi^g_{-t}|_{E_u}\|\leq Ce^{-\nu t}$
for all $t>0$; the distributions $E_s$ and $E_u$ are called the stable and unstable bundles. Another assumption we make is the absence of conjugate points: recall that 
$x,x'$ are said to be conjugate points if there is $v\in S_xM$, and $t> 0$ such that $\pi_0(\varphi_t^g(x,v))=x'$ and 
\[ 
\mathrm{d}\varphi_{t}^g(x,v)\mc{V}\cap \mc{V}\not=0,
\]
where $\mc{V}:=\ker \mathrm{d}\pi_0\subset T(SM)$ is the vertical bundle, $\pi_0:SM\to M$ being the natural projection on the base.
\begin{definition}
Let $(M,g)$ be a smooth compact, connected, Riemannian manifold with non-empty boundary.  We say that $(M,g)$ is of Anosov type if
\begin{enumerate}
\item the boundary $\pl M$ is strictly convex,
\item the trapped set $\mc{K}^g$ is a hyperbolic set for the geodesic flow $\varphi_t^g$ of $g$,
\item $(M,g)$ does not have pairs of conjugate points.
\end{enumerate}
\end{definition}
Properties $(2)$ and $(3)$ hold when $g$ has negative curvature. However this needs not be always the case. Indeed, one can either perturb a negatively curved manifold to create a small patch of positive curvature away from the trapped set, or consider a strictly convex set inside a closed manifold with Anosov geodesic flow. See \cite{Eberlein} for a discussion of such manifolds.

We prove the following result, that contains Theorem \ref{Th1}:
\begin{theorem}\label{Th2}
Let $(M_1,g_1)$ and $(M_2,g_2)$ be two real analytic Riemannian manifolds of Anosov type with analytic boundary. 
Assume that $(\pl M_1,h_1)=(\pl M_2,h_2)$ where $h_i:=g_i|_{T\pl M_i}$ is the metric on the boundary, and that their scattering data
 ${\rm Sc}(M_1,g_1)={\rm Sc}(M_2,g_2)$ agree. Then there exists an analytic diffeomorphism $\psi:M_1\to M_2$ such that $\psi|_{\pl M_1}={\rm Id}$ and 
$\psi^*g_2=g_1$.
\end{theorem}

We emphasize that we do not require to know the travel time, i.e. the lengths of geodesics between boundary points, to obtain the rigidity. 
The corresponding problem of determining the metric from both the scattering data and the travel time is called the \emph{lens rigidity problem}.  For simple manifolds, the lens ridigity, scattering rigidity and boundary rigidity problems are all equivalent, but this is not the case for non-simple manifolds;
we refer to the article of Croke--Wen \cite{Croke-Wen} for a study of the difference between scattering rigidity and lens rigidity. 

The lens rigidity problem has been considered in the non simple case. 
There are more results on lens rigidity than on scattering rigidity, although not many in the case with non-empty trapped set. Here are a few:
\begin{itemize} 
\item Stefanov--Uhlmann \cite{Stefanov:2009lp} obtained a local\footnote{Local in the sense that two metrics in that particular class that are close enough in some $C^k$ norm and that have same lens data must be isometric.} lens rigidity result for a certain class of Riemannian manifolds in dimension $n\geq 3$ that can have some mild trapping.
\item Vargo \cite{Vargo2009} proved lens rigidity in the class of  non-trapping analytic metrics, without assuming 
convexity of the boundary and in a class of metrics that can have a certain amount of conjugate points. 
\item Croke--Herreros \cite{Croke-Herreros} proved that a $2$-dimensional cylinder in negative curvature is lens rigid. 
\item Guillarmou--Mazzucchelli--Tzou \cite{Guillarmou-Mazzucchelli-Tzou} prove lens rigidity in the class of non-trapping surfaces with no conjugate points (but the boundary is not assumed convex).
\item Stefanov--Uhlmann--Vasy \cite{Stefanov:2017lt} prove lens rigidity in the class of Riemannian manifolds admitting strictly convex foliations in dimension $n\geq 3$. This class covers some cases  with conjugate points and some mild trapping.
\item Cekic--Guillarmou--Lefeuvre \cite{Cekic-Guillarmou-Lefeuvre22} prove a local lens rigidity result for  negatively curved compact manifolds with strictly convex boundary.
\end{itemize}
A real difficulty for the lens rigidity problem in the case of trapped geodesics is that having same lens data for two metrics does not a priori imply that their geodesic flows are conjugate. To avoid that problem, one can also consider the \emph{marked lens rigidity problem}, where we consider rather the lens data on the universal cover. For two negatively curved metrics on surface $M$ with strictly convex boundary, having the same marked lens data is equivalent to having conjugate geodesic flows with a conjugacy homotopic to identity and equal to the identity on the boundary $\pl SM$ of the unit tangent bundle.
In this direction, we mention the following results: 
\begin{itemize}
\item Guillarmou--Mazzucchelli \cite{Guillarmou-Mazzucchelli-18} prove that negatively curved Riemannian surfaces with strictly convex boundary having the same marked length spectrum are isometric. This has been extended recently by Erchenko-Lefeuvre \cite{Erchenko-Lefeuvre} for surfaces of Anosov type. 
\item Lefeuvre \cite{Lefeuvre-19-2} prove that negatively curved Riemannian manifolds with strictly convex boundary having the same marked length spectrum and that are close enough in $C^k$ norms are isometric.
\end{itemize}
 This rigidity problem is somehow closer to the \emph{marked length spectrum rigidity} on closed Riemannian manifolds with Anosov geodesic flow, where one asks whether two such metrics with conjugate geodesic flows (or equivalently same marked length spectrum) are actually isometric, see Otal, Croke or Guillarmou--Lefeuvre--Paternain  \cite{Otal-90,Croke-90,GLP} for a solution on surfaces and Guillarmou--Lefeuvre \cite{Guillarmou-Lefeuvre-18} for results in higher dimension. A close link between these problems on manifolds with boundary and closed manifold 
 is shown in the recent work of Chen--Erchenko--Gogolev \cite{Chen-Erchenko-Gogolev-20}. 
 It turns out that for closed manifolds, the length spectrum rigidity and the marked length spectrum rigidity problem are not equivalent, as there exist, by Vignéras \cite{Vigneras80}, Riemannian metrics with constant negative curvature that have the same length spectrum but are not isometric, while surfaces with negative curvatures and same marked length spectrum are isometric \cite{Otal-90}. The result in Theorem \ref{Th1} thus suggests that these different phenomena in the closed case
 do not seem to happen for manifolds with boundary in negative curvature, since we do not require to know the marking to obtain rigidity, while in the closed case this is necessary even in the analytic category.

In terms of difficulty of these rigidity problems, one has the chain of implications: 
\[ \textrm{scattering rigidity} \Longrightarrow \textrm{lens rigidity} \Longrightarrow \textrm{marked lens rigidity}\]
and, as mentionned above, they are not always equivalent. 

In the case of obstacles in $\R^n$, there are related lens rigidity results by Noakes--Stoyanov \cite{Noakes-Stoyanov2015} who show that the scattering sojourn times allow to recover finitely many disjoint strictly convex obstacles. In the case of analytic strictly convex obstacles with the non-eclipse condition, De Simoi--Kaloshin--Leguil \cite{Desimoi-Kaloshin-Leguil} prove that the length of the marked periodic orbits generically determines the obstacles 
under a $\Z^2\times \Z^2$ symmetry assumption (the marking is by the ordered sequence of obstacles hit by the orbit). 
This last result suggests that analytic negatively curved manifold with strictly convex boundary could also possibly be determined by their marked length spectrum (length of periodic orbits marked by their free homotopy classes).\\

To conclude, let us finally state a conjecture motivated by Theorems \ref{Th1} and \ref{Th2}:
\begin{conjecture}\label{conj2}
Two compact Riemannian manifolds of Anosov type with the same scattering data must be isometric by an isometry fixing the boundary. 
\end{conjecture}

\textbf{Method of proof}: Our proof shares some similarities with the solution of the Calder\'on problem for analytic Riemannian metrics due
to Lassas--Uhlmann \cite{Lassas-Uhlmann2001,Lassas-Taylor-Uhlman2003}. In the Calder\'on problem case, we recall that the argument is first to show that the Cauchy data for the equation $\Delta_gu=0$ on $M$ determines the Green's function of the Laplacian $\Delta_g$ near the boundary $\pl M$ when the metrics are assumed to be analytic. Then one proves that two analytic metrics with the same Green's function near $\pl M$ are isometric in dimension $n>2$ (conformal in dimension $n=2$). 

In our case, we first show that the scattering data determines the metric near the boundary (see Lemma \ref{determinationhr}), and then that it determines
the resolvent $R_g$ of the geodesic vector field $X_g$ near the boundary of $SM$ (see the proof of Lemma \ref{S_g_equal_implies_Phii_equal}). 
For this to hold, we use the analyticity, the strict convexity of $\pl M$ and the fact that the trapped set has Lebesgue measure $0$ (in order to have a proper notion of resolvent of $X_g$). The scattering data are essentially the Cauchy data at $\pl(SM)$  for the equation $X_gu=0$ on $SM$.

Compared to the Calder\'on problem where the Green's function is a pseudo-differential operator (as inverse of an elliptic differential operator), 
the resolvent $R_g$ is a much more singular operator. It was analyzed for smooth metrics with hyperbolic trapped set in \cite{DG16} using microlocal methods. Its integral kernel has quite wild singularities at the trapped set, but the $C^\infty$ wavefront set can still be described (even though it typically looks like a fractal conical set in $T(SM)$).  Composing by the pull-back operator $\pi_0^*$ and the push-forward operator ${\pi_0}_*$ (where $\pi_0:SM\to M$), 
using the transversality between the vertical and stable/unstable bundles, and the absence of conjugate points, we obtain an operator $\Pi^g_0:={\pi_0}_*R_g\pi_0^*$ that is much better behaved. Indeed, most of the singularities disappear and $\Pi^g_0$  is an elliptic pseudo-differential operator of order $-1$ with principal symbol  $c_n|\xi|^{-1}_g$ for some $c_n>0$ depending only on $n=\dim(M)$, as was proved in \cite{Guillarmou17}.
 It thus looks quite similar to the Green's function of the Laplacian (except for its order that is half of that of the Green's function). 
 
We can then use the techniques from analytic microlocal analysis developped in \cite{BJ20} (building from fundamental works of 
Helffer-Sj\"ostrand and Sj\"ostrand \cite{helffer_sjostrand,sjostrand_96}, see the work of Galkowski and Zworski \cite{GaZw_viscosity,GaZw} for another recent example of application of these methods) to analyze the Schwartz kernel $\Pi^g_0(x,x')$ of $\Pi^g_0$:
we are able to show in Section \ref{sec:FBIresults} that, if $g$ is in addition analytic, the Schwartz kernel $\Pi^g_0(x,x')$ of $\Pi^g_0$ is analytic outside the diagonal. This requires in particular to prove radial estimates for the analytic wave front set (Proposition \ref{prop:radial_estimate}); a related result was proved by Galkowski-Zworski \cite{GaZw} but rather for smooth radial manifolds.
The fact that the operator $\Pi_0^g$ is an analytic pseudo-differential operator 
 was first proved for simple analytic metrics by Stefanov and Uhlmann \cite{Stefanov-Uhlmann-JAMS2005}. In our case, the analysis involved is significantly more consequent due to the trapped set and the result can not be deduced from  \cite{Stefanov-Uhlmann-JAMS2005}; 
 we rely on the recent monograph \cite{BJ20}.
We can then use a method developped (for the Green's function) in \cite{Lassas-Taylor-Uhlman2003} to embed analytically $(M,g)$ into $L^2(N_\eps)$ using $\Pi_0$, where $N_\eps$ is an $\eps$-collar neighborhood of $N:=\pl M$, and show that the image of the embedding in $L^2(N_\eps)$ is entirely determined by the scattering data. These embeddings for two metrics with the same scattering data allow to construct an isometry $\psi:  M_1\to M_2$. 

Another corollary of this work is the following (see Proposition \ref{injectiviteI2} for a precise statement) 
\begin{proposition}
On a real analytic Riemannian manifold of Anosov type, the X-ray transform on symmetric $2$-tensors is solenoidal injective. 
\end{proposition}

{\bf Notations}:  $\mc{A}(M)$ denotes the set of analytic functions on an analytic manifold $M$ with or without boundary,  
$\mc{D}'(M)$ denotes the set of distributions on a manifold $M$ without boundary, defined as the topological dual to the space $C_c^\infty(\mathring{M})$ of smooth compactly supported functions in the interior $\mathring{M}$, $H^s(M)$ denotes the $L^2$-based Sobolev space of order $s$ on a compact manifold with or without boundary. We shall denote by $d_{M}(\cdot,\cdot)$ a fixed distance on $M$, if $M$ is a compact manifold.\\

{\bf Acknowledgements}: This project has received funding from the European Research Council (ERC) under the European Union’s Horizon 2020 research and innovation programme (grant agreement No. 725967). During the beginning of this project, M.J. was supported by the European Research Council (ERC) under the European Union’s Horizon 2020 research and innovation programme (grant agreement No 787304) and was working at LPSM\footnote{Laboratoire de Probabilités, Statistique et Modélisation (LPSM), CNRS, Sorbonne Université, Université de Paris, 4, Place Jussieu, 75005, Paris, France}. Y.G.B is supported by the Agence Nationale de la Recherche through the PRC grant ADYCT (ANR-20-CE40-0017).
We thank Y. Canzani for Figure \ref{scatteringmap}.

\section{Scattering map and the normal operator}

\subsection{Dynamical preliminaries}

In this section, we introduce several important dynamical preliminaries, we refer to  \cite{Guillarmou17} for a more detailed exposition.\\
 
Let $(M,g)$ be a real-analytic Riemannian manifold with analytic strictly convex boundary $\pl M$. We denote by $SM$ its unit tangent bundle, which is also an analytic manifold with analytic boundary, and $\pi_0:SM\to M$ the projection on the base. 
We shall use  $z=(x,v)$ for the variable on $SM$, where $x=\pi_0(z)$ and $v\in S_xM$ is the tangent vector, and we denote by $\alpha_L$ the Liouville $1$-form on $SM$ defined by $\alpha_{L}(x,v)(V)=g_x(v,\mathrm{d}\pi_0(V))$ for $V\in T_{(x,v)}(SM)$. 
Finally, we denote by $-z:=(x,-v)$ the reflected vector. 

The geodesic vector field $X_g$ has analytic coefficients and its flow is denoted 
by $\varphi^g_t:SM\to SM$. The boundary $\pl SM$ splits into 
\[\pl SM= \pl_-SM\sqcup \pl_+SM\sqcup \pl_0 SM\]
where $\pl_\pm SM:=\{(x,v)\in  \pl SM\,|\, \mp g_x(v,\nu)>0\}$  and $\pl_0SM=\{(x,v)\in \pl SM\,|\ g_x(v,\nu)=0 \}$, where $\nu$ denotes the inward unit normal vector to $\pl M$. Define the measure on $\pl_\pm SM$
\[ 
\mathrm{d}\mu_\nu(x,v)=|\iota_{\pl_\pm SM}^*i_{X_g}\omega_L|
\]
where $\omega_L=\alpha_{L}\wedge \mathrm{d}\alpha_{L}^{n-1}$ is the Liouville volume form, and $\mathrm{d}\mu_L:=|\omega_L|$ the Liouville measure.
 
We define the escape time by 
\[ \tau_g: SM\to \R_+\cup \{\infty\} , \quad \tau_g(z)=\sup \{t\geq 0 \, | \, \forall s\in ]0, t], \varphi^g_s(z)\in SM \}.\]
One has $\tau_g=0$ on $\pl_+SM\cup \pl_0SM$ by strict convexity of the boundary.
The forward (resp. backward) trapped set $\Gamma^g_-$ (resp. $\Gamma_+^g$) are defined by 
\begin{equation}\label{def:Gammapm} 
\Gamma^g_\pm :=\{ z\in SM\, |\, \tau_g(\mp z)=\infty\}.
\end{equation}
By the analytic implicit function theorem and the fact that $\pl M$ is strictly convex, the map $\tau_g$ belongs to $\mc{A}(SM\setminus (\Gamma_-^g\cup \pl_0SM))$.

The trapped set is defined by 
\begin{equation}\label{def:trappedset}
\mc{K}^g:= \Gamma^g_+\cap \Gamma^g_-.
\end{equation}
It is a closed subset of $S\mathring{M}$ invariant by the flow $\varphi_t$. We assume that $\mc{K}^g$ is hyperbolic. In that case, ${\rm Vol}(\Gamma^g_\pm)=0$, see \cite[Section 2]{Guillarmou17}. We define the dual bundles $E_0^*,E_s^*,E_u^*\subset T^*(SM)$  over $\mc{K}^g$ by 
 \[ E_0^*(E_u\oplus E_s)=0, \quad E_u^*(E_u\oplus \R X_g)=0, \quad E_s^*(E_s\oplus \R X_g)=0.\]
By \cite[Lemma 1.11]{DG16}, the bundle $E_s^*$ over $\mc{K}^g$ admits a continuous extension $E_-^*$ to the set $\Gamma_-^g$, and $E_u^*$ admits a continuous extension $E_+^*$ to the set $\Gamma_+^g$, and they satisfy 
\[ 
E_\pm^*(\R X_g)=0. 
\]

As explained in \cite[Section 2]{Guillarmou17}, since $(M,g)$ has hyperbolic trapped set and no conjugate points, there is a small extension $(M_e,g_e)$ of $(M,g)$ which has the same properties,  with the same trapped set $\mc{K}^{g_e}=\mc{K}^{g}$. The sets $\Gamma_\pm^{g_e}$ associated with $g_e$ are extensions of $\Gamma_\pm^g$, and we still denote by $E_\pm^*$ the bundles associated with $g_e$, which are extensions of those associated with $g$.
Finally, since $(M,g)$ is analytic with analytic boundary, one can choose $(M_e,g_e)$ to be also analytic with analytic boundary.
 
 Next we define the scattering map.
 \begin{definition}
The scattering map $S_g: \pl_-SM\setminus \Gamma^g_-\to \pl_+SM\setminus \Gamma^g_+$ is defined by 
\[
S_g(z):=\varphi^g_{\tau_g(z)}(z).
\] 
The scattering operator $\mc{S}_g: C_c^\infty(\pl_+SM\setminus \Gamma^g_+)\to C_c^\infty(\pl_-SM\setminus \Gamma^g_-)$ is defined by 
\[
\mc{S}_gf(z)= f(S_g(z)).
\]
By \cite[Lemma 3.4]{Guillarmou17}, it extends as a unitary map from $L^2(\pl_+SM,\mathrm{d}\mu_\nu)$ to $L^2(\pl_-SM,\mathrm{d}\mu_\nu)$.
 \end{definition}
 
We also define the backward resolvent for the flow 
\begin{equation}\label{defRg} 
R_g: C_c^\infty(S\mathring{M}\setminus \Gamma_-^g)\to C^\infty(SM), \quad R_gf(z)=\int_{-\tau_g(-z)}^0f(\varphi^g_{t}(z))\mathrm{d}t.
\end{equation}
This operator yields a solution to the boundary value problem 
\[
X_gR_gf=f,  \quad R_gf|_{\pl_-SM}=0,\quad \forall f\in C_c^\infty(S\mathring{M} \setminus \Gamma^g_-).
\]
According to \cite[Proposition 4.2]{Guillarmou17}, $R_g$ extends as a bounded map
\begin{equation}\label{extensionRg}
R_g :L^\infty(SM)\to L^p(SM), \quad \forall p\in [1,\infty),
\end{equation}
and $X_gR_g={\rm Id}$ in $\mc{D}'(S\mathring{M})$ with $(R_gf)|_{\pl_-SM}=0$ if $f\in C_c^\infty(S\mathring{M})$. One can also define $R_{g_e}$ the resolvent for the extension $(M_e,g_e)$, and it is direct to see, by convexity of $\pl M$, that $(R_{g_e}f)|_{SM}=R_gf$ if $\supp(f)\subset SM$.
A closely related operator is the \emph{X-ray transform}, which is defined for $f\in C_c^\infty(SM\setminus (\Gamma_+^g \cup \Gamma_-^g))$ by the formula
\begin{equation}\label{eq:def_Xray}
I_gf(y):=R_gf(y)=\int_{-\tau_g(-y)}^{0}f(\varphi^g_{t}(y))\mathrm{d}t \in C_c^\infty(\pl_+SM\setminus\Gamma^g_+ ).
\end{equation}
It satisfies $I_gX_gf=0$. Its adjoint using the $\mathrm{d}\mu_\nu$ and $\mathrm{d}\mu$ measures is denoted $I_g^*$, satisfies $X_gI_g^*=0$ on $\mc{D}'(\pl_+SM\setminus\Gamma_+^g)$. For $u\in C_c^\infty(\pl_+SM\setminus\Gamma_+^g )$, it is given for $z\in SM\setminus \Gamma_-^g$ by 
\begin{equation}\label{Ig*formula}
I_g^*u(z)=u(\varphi^g_{\tau_g(z)}(z)).
\end{equation}

 By \cite[Lemma 5.1]{Guillarmou17} the operator $I_g$ extends as a continuous map $I_g:L^p(SM)\to L^2(\pl_-SM)$ for all $p>2$.

By abuse of notation we also denote by $(z,z')\mapsto R_g(z,z')$ (resp.$(z,z')\mapsto R_{g_e}(z,z')$) the Schwartz kernel 
of $R_g$ (resp. $R_{g_e}$) which is viewed as an element in $\mathcal{D}'(S\mathring{M} \times S\mathring{M})$ (resp. $\mathcal{D}'(S\mathring{M}_e \times S\mathring{M}_e)$). We have $R_g(\cdot,\cdot)=R_{g_e}|_{ S\mathring{M}\times S\mathring{M}}$. The Schwartz kernel of $I_g$ is $R_{g_e}|_{\pl_+SM \times S\mathring{M}}$.
 
We finally remark that the scattering map has a Schwartz kernel given by the restriction of the resolvent kernel of $X_g$ to $\pl_\pm SM$: 
\[ 
\mc{S}_g = R_{g_e}|_{\partial_+ SM \times \partial_- SM}.
\] 
(See \cite[Proposition 3.2]{Chaubet-21} in dimension 2 and \cite{Cekic-Guillarmou-Lefeuvre22} in general). 
In comparison with the Calder\'on problem, this is similar to the relation between the Dirichlet-to-Neumann map and the Green kernel of the Laplacian with Dirichlet condition.

We close this section by a description of the $C^\infty$ wave front set of the Schwartz kernel of $R_{g_e}$, taken from \cite{DG16}. This will be important in section \ref{sec:FBIresults}. We have
\begin{equation}\label{WFR} 
\begin{gathered}
{\rm WF}(R_{g_e})\subset N^*\Delta(SM_e) \cup \Omega_+(SM_e)\cup (E_+^*\x E_-^*),
\end{gathered}
\end{equation} 
where $N^*\Delta(SM_e)$ is the conormal bundle to the diagonal 
$\Delta(SM_e)$ of $SM_e\x SM_e$,  
\[
\begin{split}
\Omega_\pm(SM_e):= \{(\varphi^g_{t}(z),(d\varphi^g_{t}(z)^{-1})^T\xi ;\ z,-\xi)\in \, & T^*(SM_e\x SM_e)\, | \,  \\ &  \,\, \pm t\geq 0,\,\, \xi(X_g(z))=0\}.
\end{split}
\] 
The dual resolvents $R_g^*$ and $R_{g_e}^*$  are also well-defined (for example as maps $L^2\to L^1$ by \eqref{extensionRg}) and have Schwartz kernel 
\begin{equation}\label{integralkerneldual}
 R_g^*(z,z')=R_g(z',z), \quad R_{g_e}^*(z,z')=R_{g_e}(z',z).
 \end{equation}
They are given by the expressions 
\[ 
R_g^*f(z)=\int_{0}^{\tau_g(z)}f(\varphi_t^g(z))\mathrm{d}t, \quad  R_{g_e}^*f(z)=\int_{0}^{\tau_{g_e}(z)}f(\varphi_t^{g_e}(z))\mathrm{d}t,
\]
they satisfy $-X_gR_g^*f=f$ in $\mc{D}'(S\mathring{M})$ with $R_g^*f|_{\pl_+SM}=0$ for $f\in C^\infty(SM)$ 
(resp. $-X_{g_e}R_{g_e}^*f=f$ in $S\mathring{M}_e$ if $f\in C^\infty(SM_e)$) and the wave-front set of the Schwartz kernel of $R_{g_e}^*$ satisfies 
\begin{equation}\label{WFR*}
{\rm WF}(R_{g_e}^*)\subset N^*\Delta(SM_e) \cup \Omega_-(SM_e)\cup (E_-^*\x E_+^*).
\end{equation}

\subsection{Determination of the metric near the boundary}

We start with a normal form for the metric near the boundary:
\begin{lemma}\label{normalform}
Let $(M,g)$ be an analytic Riemannian manifold with analytic strictly convex boundary. There exists an analytic diffeomorphism 
 $\psi:[0,\epsilon)_r \times \pl M\to U\subset M$, called normal form for $g$, such that
\[ \psi^*g=\mathrm{d}r^2+h(r)\]
where $r\mapsto h(r)\in \mc{A}(\partial M;S^2T^*M)$ is a one parameter analytic family of analytic metrics on $\pl M$.
\end{lemma}
\begin{proof}
It suffices to use the map 
$\psi(r,y)=\exp_{y}(r\nu(y))$ for $\nu$ the inward unit normal vector to $\pl M$. This map is analytic map since $g$ is analytic. 
\end{proof}

Using a standard argument, we show that $S_g$ determines the Taylor expansion of the map $r\mapsto h(r)$ at $r=0$ and so it determines $h(r)$ for $r$ near $0$ by unique continuation:
\begin{lemma}\label{determinationhr}
Assume that $(M_1,g_1)$ and $(M_2,g_2)$ are two analytic Riemannian manifolds with strictly convex analytic boundary and no conjugate points, and let $\psi_1$ and $\psi_2$
their normal form diffeomorphisms and denote by $\psi_i^*g_i=\mathrm{d}r^2+h_i(r)$. Assume that $(\pl M_1,h_1(0))=(\pl M_2,h_2(0))$ 
and that ${S}_{g_1}={S}_{g_2}$ (i.e. $M_1$ and $M_2$ have the same scattering data). Then $h_1(r)=h_2(r)$ near $r=0$.
\end{lemma}
\begin{proof} 
Let us denote by $(N,h):=(\pl M_1,h_1(0))=(\pl M_2,h_2(0))$. Let $b_{i}\in C^0(N\times N)$ be the boundary distance function defined by 
$b_{i}(y,y'):=d_{g_i}(y,y')$ where $d_{g_i}\in C^0(M_i\times M_i)$ is the Riemannian distance of $(M_i,g_i)$. 
First, we will show that the identity ${S}_{g_1}={S}_{g_2}$ implies that the boundary distances $b_{1}$ and $b_{2}$ for $g_1$ and $g_2$ agree near the diagonal of $N\times N$. Then, we can apply Theorem 2.1 of Lassas--Sharafutdinov--Uhlmann \cite{Lassas:2003}, which says that  $h_1(r)-h_2(r)=\mc{O}(r^\infty)$ at $r=0$ (since their argument is purely local near the boundary) and deduce that $h_1(r)=h_2(r)$ by analyticity.
 
To show that $b_{1}=b_{2}$ near the diagonal if ${S}_{g_1}={S}_{g_2}$, first choose $\eps>0$ small and denote the injectivity radius of $g_i$ by $r_i$, $i=1,2$. If $x,x'\in N$ are such that their  distance in $(N,h)$ satisfies $d_{h}(x,x')<\eps$ with $\eps$ small enough, then $d_{g_i}(x,x')<r_i$. Pick such $x,x'\in N$. 
Take the minimal geodesic curve $\gamma:[0,1]\to N$ in $(N,h)$ so that $\gamma(0)=x$ and $\gamma(1)=x'$, in particular $d_{g_i}(\gamma(t),x)<\eps$. We define $v_i(t)$ to be the unit vector in the direction of $\gamma(t)$, that is $(\exp_{x}^{g_i})^{-1}(\gamma(t))/|(\exp_{x}^{g_i})^{-1}(\gamma(t))|_{g_{i,x}}\in S_xM$, so that $b_i(x,\gamma(t))=\tau_{g_i}(x,v_i(t))$ (we extend $v_i(t)$ by continuity at $t = 0$). 

We claim that $v_i(t)$ is the unique vector that has the smallest orthogonal projection to $(TN)^\perp\subset TM$ among all vectors $(x,v)\in \pl_-SM_i$ satisfying $\pi_0(S_{g_i}(x,v))=\gamma(t)$, i.e
\[
 |g_i(v_i(t),\nu_i)|=\min \big\{ |g_i(v,\nu_i)|\, \big| \, v\in \pl_-S_xM_i,  \pi_0(S_{g_i}(x,v))=\gamma(t) \big\},
\]
 with $\nu_i$ the inward unit normal at $N$ in $M_i$. In particular this claim implies that $v_i(t)$ is entirely determined by $h_i(0)$ and ${S}_{g_i}$, and since 
\[
b_i(x,x')=\int_0^1 \pl_t( d_{g_i}(x,\gamma(t)))dt=\int_0^1 g_i(\nabla^{g_i}d_{g_i}(x,\cdot)|_{\gamma(t)},\dot{\gamma}(t))dt,
\]
wih $\nabla^{g_i}d_{g_i}(x,\cdot)|_{\gamma(t)}=S_{g_i}(x,v_i(t))$ (by Gauss Lemma), we conclude that $b_1=b_2$ if ${S}_{g_1}={S}_{g_2}$.

Let us prove the claim above. We start by observing that the projection $|g_i(v_i(t),\nu_i)|$ is bounded by $C \epsilon$ for some $C > 0$ that does not depend on $\eps$. 
Let us now consider other geodesic curves with endpoints $x$ and $\gamma(t)$. Due to the absence of conjugate points, there is exactly one such curve in each homotopy class of curves with endpoints $(x,\gamma(t))$, so if $\tilde{v}_i(t)\neq v_i(t)$ is such that $\pi_0(S_{g_i}(x,\tilde{v}_i(t))) = \gamma(t)$, the corresponding curve cannot be homotopic to $\gamma$ and therefore its length $\tau_{g_i}(x,\tilde{v}_i(t))$ must be larger than $r_i$.  
On the other hand,  by strict convexity of the boundary, there is $C>0$ such that for all $\delta>0$ small and $(x,v)\in \pl_-SM_i$ such that $|g_i(v,\nu_i)|<\delta$,  
 one has $ \tau_{g_i}(x,v)\leq C|g_i(v,\nu_i)|$ (\cite[Lemma 4.1.2.]{Sharfutdinov-Book}). We deduce that $|g_i(\tilde{v}_i(t),\nu_i)|> \delta$ since otherwise one would have 
 \[ 
 r_i\leq \tau_{g_i}(x,\tilde{v}_i(t))\leq C|g_i(\tilde{v}_i(t),\nu_i)| \leq C \delta
 \]
 which is a contradiction if $\delta>0$ is small enough. Taking $\epsilon$ small enough so that $C \epsilon < \delta$, the claim is then proved and the proof is complete.
\end{proof}

\subsection{The normal operator}

Let us now define the normal operator $\Pi^g_0$ that was introduced in \cite{Guillarmou17}. Let $\pi_0^*:C^\infty(M)\to C^\infty(SM)$ be the pullback by $\pi_0:SM\to M$ (the projection on the base). It is continuous as a map $\pi_0^*:\mc{A}(M)\to \mc{A}(SM)$. Consider also the dual map, called pushforward, ${\pi_0}_*:\mc{D}'(S\mathring{M})\to \mc{D}'(\mathring{M})$ and we use the same notations for the corresponding operators on the extended manifolds $SM_e$  and $M_e$. Then we define 
\[ \Pi^g_0 : C^{\infty}_c(\mathring{M})\to \mc{D}'(\mathring{M}), \quad \Pi^g_0:={\pi_0}_*R_g\pi_0^*.\]
The same definition holds on the extension $M_e$ for the corresponding operator $\Pi_0^{g_e}$, and since $R_g$ is the restriction of $R_{g_e}$ on functions supported in $SM$, we directly deduce that $(\Pi_0^{g_e}f)|_M=\Pi_0^gf$ if $\supp(f)\subset \mathring{M}$. Finally, using that $\pi_0(x,-v)=\pi_0(x,v)$, 
one also has $ \Pi^{g_e}_0={\pi_0}_*R_{g_e}^*\pi_0^*$ and thus (this fact also follows from the identity $\Pi_0^{g_e}={\pi_0}_*{I^{g_e}}^*I^{g_e}\pi_0^*$, see \cite[Section 5.1]{Guillarmou17})
\begin{equation}\label{eq:symetrique}
(\Pi_0^{g_e})^*=\Pi_0^{g_e}.
\end{equation}
The operator $I_0^{g_e}:=I^{g_e}\pi_0^*$ is called the X-ray transform on $M$, it is injective by \cite[Theorem 5]{Guillarmou17}.  

We recall from  \cite[Proposition 5.7]{Guillarmou17}:
\begin{lemma}\label{lem:asymptPi0}
If $(M,g)$ is a smooth manifold of Anosov type, 
then $\Pi_0^g\in \Psi^{-1}(\mathring{M})$
 is an elliptic pseudo-differential operator of order $-1$ on $\mathring{M}$ with principal symbol 
$\sigma(\Pi_0^g)=c_n|\xi|_g^{-1}$ for some $c_n>0$ depending only on $n=\dim M$. It is moreover the restriction to $M$ of the elliptic pseudo-differential operator $\Pi_0^{g_e}\in \Psi^{-1}(\mathring{M}_e)$.
For each $x_0\in \mathring{M}_e$, its Schwartz kernel has asymptotic behaviour as $x\to x_0$
\begin{equation}\label{asymptPi0}
 \Pi_0^{g_e}(x,x_0)\sim c_n'd_{g_e}(x,x_0)^{-(n-1)}, \quad c_n'\not=0
 \end{equation}
where $d_{g_e}$ denotes the Riemannian distance on $(M_e,g_e)$. Similarly for any $v\in T_{x_0}M$, we have as $x\to x_0$
\begin{equation}\label{asymptDPi0}
 D_{x_0}\Pi_0^{g_e}(x_0,x).v\sim -(n-1)c_n' \Big(D_{x_0}(d_{g_e}(x,x_0)).v\Big) d_{g_e}(x,x_0)^{-n}.
 \end{equation}
\end{lemma}

If in addition the metric is analytic, we show the following stronger result, the proof of which is deferred to Section \ref{sec:FBIresults}. It is the main technical result of the paper. 
\begin{proposition}\label{analyticityofPi0}
If $(M_e,g_e)$ is an analytic manifold of Anosov type, then $(x,x')\mapsto \Pi^{g_e}_0(x,x')$ is analytic in $\mathring{M}_e \times \mathring{M}_e \setminus \{\rm diag\}$.
\end{proposition}

Even though it is not required for the proof of Theorem \ref{Th2}, we can use the local analysis from \cite{Stefanov-Uhlmann-JAMS2005} to get a more precise statement.

\begin{proposition}\label{proposition:analytic_pseudo}
The operator $\Pi_0^{g_e}$ is an elliptic analytic pseudo-differential operator of order $-1$ on $\mathring{M}_e$ (in the sense of \cite[Chapter V]{treves-80}).
\end{proposition}

We give another consequence of the analysis performed in Section \ref{sec:FBIresults}, namely that the X-ray transform on real analytic 
manifolds  of Anosov type  is injective on the space of divergence-free symmetric $2$-tensors. First, define 
the operator mapping a symmetric  tensor of order $m\geq 1$ on $M_e$ to a function on $SM_e$ 
\[ \pi_m^*: C_c^\infty(\mathring{M}_e;S^mT^*\mathring{M}_e)\to C_c^\infty(S\mathring{M}_e), \quad (\pi_m^* f)(x,v):=f_x(\otimes^mv).\]
This operator maps real analytic tensors to real analytic functions on $SM_e$. We define ${\pi_m}_*:\mc{D}'(S\mathring{M}_e)\to \mc{D}'(\mathring{M}_e;S^mT^*\mathring{M}_e)$ to be its dual defined by $\cjg {\pi_m}_*u,\chi\cjd:=\cjg u,\pi_m^*\chi\cjd$ for $\chi\in C_c^\infty(\mathring{M}_e;S^mT^*\mathring{M}_e)$. 
We define the X-ray transform on symmetric $m$-tensors on $(M_e,g_e)$ by 
\[I_m^{g_e} := I^{g_e}\pi_m^*.\]
We obtain the same result as \cite[Theorem 1]{Stefanov-Uhlmann-JAMS2005}  and \cite[Theorem 1]{Stefanov-Uhlmann08}, but 
now in the setting of manifolds of Anosov type.
\begin{proposition}\label{injectiviteI2}
Assume that $(M,g)$ is  a real analytic Riemannian manifold of Anosov type. Let $f\in L^2(M;S^2T^*M)$ and assume that $I_2^{g}f=0$. Then  $f=D_gh$ for some one-form $h\in H_0^1(M;T^*M)$ vanishing at $\pl M$, where $D_g$ is the symmetrized Levi-Civita covariant derivative on tensors.
\end{proposition}
The proof is defered to Section  \ref{sec:FBIresults}. We notice that it can be extended with minor modifications to $m$-tensors for all $m\geq 1$. We focus on the case $m=2$ since we use some results from \cite{Stefanov-Uhlmann-JAMS2005,Stefanov-Uhlmann08} done for $m\leq 2$ and since this is the case of interest for lens and boundary rigidity problems.

\section{Proof of Theorem \ref{Th2}}
We consider two analytic extensions $(M_e^1,g_e^1)$ and $(M_e^2,g_e^2)$ of $M_1$ and $M_2$ respectively. Let us denote by $N:=\pl M_1=\pl M_2$ the common boundary of $M_1,M_2$. Let $\psi_1:[0,\eps)\times N \to U_1\subset M_e^1$ and $\psi_2:[0,\eps)\times N\to U_2\subset M_e^2$ 
be the two normal forms near the boundary for $g_e^1,g_e^2$ (for $\epsilon>0$ small enough), i.e. 
\[
\psi_i^*g_e^i=\mathrm{d}r^2+h_i(r)
\] 
with $h_i$ an analytic family of analytic metrics on $N:=\pl M_1=\pl M_2$. By choosing the extension small enough, we can assume 
that $\psi_i(\{\eps/2\}\times N)=\pl M_i$ and that $\psi_i([0,\eps) \times N)$ does not intersect $\pi_0(\mathcal{K}^{g_i})$.
Recall that $\psi_i$ is analytic, and by Lemma \ref{determinationhr}, $h_1=h_2$ in $r\in [\eps/2,\eps)$, thus everywhere in $[0,\eps)\times N$ by analytic continuation. For $0\leq\delta\leq\epsilon$, we denote $N_\delta:=(0,\delta)\times N$.

Let us now consider the map 
\begin{equation}\label{Phi}
\Phi_i: M_i\to L^2( N_{\epsilon/4} ), \quad \Phi_i(x)=\Pi_0^{g_e^i}(x,\psi_i(\cdot))|_{N_{\eps/4}}.
\end{equation}
where $\Pi_0^{g_i}(x,x')$ is the Schwartz kernel of the operator $\Pi_0^{g_e^i}$ associated with $g_e^i$. Here, since $x\in M_i$ and $\psi_i(N_{\epsilon/4}) \cap M_i = \emptyset$, the map $\Phi_i$ is valued into analytic functions on the small collar $N_{\eps/4}$. The choice of the space $L^2(N_{\eps/4})$ is irrelevant; it only needs to be a reasonable functional Banach space. The map $\Phi_i$ is analytic on $M_i$ as a composition of analytic maps, using that $\Pi_0^{g_e^i}(x,x')$ is analytic outside the diagonal.\\

 We first show the following:
\begin{lemma}\label{S_g_equal_implies_Phii_equal} 
If ${S}_{g_1}={S}_{g_2}$, then $\Phi_1(\psi_1(x))=\Phi_2(\psi_2(x))$ for each $x\in (\eps/2,\epsilon)\times N$.
\end{lemma}

\begin{proof} 
For $i=1,2$, let $\Psi_i:SN_\epsilon \to SM_i$ be the lift of $\psi_i$ to the unit tangent bundles, i.e. $\Psi_i(x,v)=(\psi_i(x),\mathrm{d}\psi_i(x)v)$.  
We will show that $(\Psi_1\otimes \Psi_1)^*R_{g_e^1}=(\Psi_2\otimes \Psi_2)^*R_{g_e^2}$ on $SN_\eps^2$ as distributions. Since (on the level of kernels)
\[
\Pi_0^{g_e^i} = (\pi_{0\ast}\otimes \pi_{0\ast}) R_{g_e^i},
\]
the announced claim follows. 

We will write $\partial_{\sigma} SN_\epsilon = \Psi_1^{-1} \partial_\sigma SM_e^1 = \Psi_2^{-1} \partial_\sigma SM_e^2$ for $\sigma \in \{+,-,0\}$. For $\chi\in C_c^\infty(S\mathring{N}_\epsilon)$, call $\chi_i:={\Psi_i}_*\chi$ and  consider the distribution 
\[ 
u_i(z)=  \cjg R_{g_e^i}(\Psi_i(z),\Psi_i(\cdot)), \chi \cjd=(R_{g_e^i}\chi_i)(\Psi_i(z)).
\]
It suffices now to prove that $u_1 = u_2$ for any such $\chi$. A priori, by Proposition 4.2 of \cite{Guillarmou17}, $u_i \in L^p(SN_\epsilon)$ for all $p<\infty$. Additionnally, according to the support of $\chi_i$, we have $u_i \in C^\infty(SN_\epsilon\setminus \Psi_i(\Gamma^{g_e^i}_+))$, and it satisfies 
\[ 
X_iu_i=\chi , \quad u_i|_{\pl_-SN_\epsilon}=0
\]
where $X_i$ is the geodesic vector field of $\mathrm{d}r^2+h_i(r)$.  We have $X_1=X_2$ in $SN_\epsilon$ and let $\varphi_t$ be the flow of $X_1=X_2$ there. From the strict convexity of the boundary, if $z\in SN_\epsilon$, we deduce that $\pi_0(\varphi_t(z))$ remains in $N_\epsilon$ in at least one of the intervals $[-\tau_{g_e^1}(\Psi_1(-z)), 0]$, $[0,\tau_{g_e^1}(\Psi_1(z))]$. In the first case, from the expression \eqref{defRg} we deduce directly that $u_1(z)=u_2(z)$. 

Since $\Gamma_+^{g_e^i}$ has measure $0$, and $u_i$ is in $L^p$, we concentrate thus on the $z$'s whose trajectory remains in $N_\epsilon$ for $t\in[0,\tau_{g_e^1}(\Psi_1(z))]$, but for which $\tau_{g_e^1}(\Psi_1(-z))< \infty$. According to the expression
\[
u_i(z)=I_{g_e^i}\chi_i(\varphi_{\tau_{g_e^i}(\Psi_i(z))}^{g_e^i} ( \Psi_i(z)))-\int_{0}^{\tau_{g_e^i}(\Psi_i(z))}\chi(\varphi_t(z))\mathrm{d}t,
\]
it suffices thus to prove that the quantity $I_{g_e^i}\chi_i(\varphi_{\tau_{g_e^i}(\Psi_i(z))}^{g_e^i} ( \Psi_i(z)))$, or equivalently $u_i|_{\partial_+ SN_\epsilon}$, does not depend on $i$ (here we use the X-ray transform $I_{g_e^i}$ defined in \eqref{eq:def_Xray}). However, the quantity $I_{g_e^i}\chi_i(\varphi_{\tau_{g_e^i}(\Psi_i(z))}^{g_e^i} ( \Psi_i(z)))$ is the sum of two terms, the first corresponds to the trajectory of $\Psi_i(z)$ close to $\Psi_i(z)$, and the second to the trajectory near $\partial_- SM_e^i$. This is because $\chi$ is supported in $N_\epsilon$. We already know that the first contribution does not depend on $i$. For the second contribution, we can use the main assumption that ${S}_{g_1}={S}_{g_2}$ to deduce that they are also equal. Finally, $u_1(z)= u_2(z)$ on $SN_\epsilon$ as expected
\end{proof}

We next show the 
\begin{lemma}\label{embedding}
The map $\Phi_i$ is an analytic embedding of $M_i$ into $L^2(N_{\eps/4})$ for $i=1,2$.
\end{lemma}
\begin{proof}
First, we show that $\Phi_i$ is an injective map. If $x_0,x_1\in M_i$ are such that $\Phi_i(x_0)=\Phi_i(x_1)$, then $\Pi_0^{g_e^i}(x_0,\psi_i(x'))=\Pi_0^{g_e^i}(x_1,\psi_i(x'))$ for all $x'\in N_{\eps/4}$. In particular, this implies that $\Pi_0^{g_e^i}(x_0,x)=\Pi_0^{g_e^i}(x_1,x)$
for all $x\in \psi_i(N_{\eps/4})$ and by analytic continuation we get $\Pi_0^{g_i}(x_0,x)=\Pi_0^{g_i}(x_1,x)$ for all $x\in \mathring{M}_e^i \setminus \{x_0,x_1\}$.
Using the asymptotic expansion \eqref{asymptPi0} and letting respectively $x$ tend to $x_0$ and to $x_1$, we deduce that $x_0=x_1$.

Next, we show its derivative $D\Phi_i(x)$ is injective for $x\in M_i$. The map $\Phi_i$ is $C^1$ on $M_i$ since $\Pi^{g_e^i}_0$ has a smooth kernel outside the diagonal (it is in fact analytic). Let $v\in T_xM_i$, and assume that $D\Phi_i(x).v=0$. We then get $D_x\Pi^{g_e^i}_0(x,x').v=0$ for all 
$x'\in M_e^i\setminus M_i$. Using the analytic extension of $\Pi^{g_e^i}_0(x,x')$ outside the diagonal we obtain  $D_x\Pi^{g_e^i}_0(x,x').v=0$ for all $x'\in M_e^i\setminus \{x\}$. By letting $x'\to x$ and using the asymptotic expansion \eqref{asymptDPi0}, we deduce that $v=0$, showing that $\Phi_i$ is an immersion, and thus an embedding.
\end{proof}

We can now conclude the proof of Theorem \ref{Th2}.

\begin{proposition}\label{sameimage}
The image of $\Phi_1$ and $\Phi_2$ agree, i.e. $\Phi_1(M_1)=\Phi_2(M_2)$, and the map $\Phi=\Phi_2^{-1}\circ \Phi_1:M_1\to M_2$ is an analytic map such that $\Phi^*g_2=g_1$.
\end{proposition}

\begin{proof} 
Let us define 
\[
U:=\{x\in M_1\,|\, \exists x'\in M_2, \Phi_2(x')=\Phi_1(x)\}.
\]
Since, according to Lemma \ref{embedding}, $\Phi_2$ is injective, we can define the map $\Phi = \Phi_2^{-1} \circ \Phi_1$ on $U$. Let $U' \subset U$ be the set of points $x \in U$ such that $\Phi$ is well-defined and $C^1$ on a neighbourhood of $x$. Notice, by Lemma \ref{S_g_equal_implies_Phii_equal}, that $\Psi_1((\eps/2,\eps)\times N)\cap M_1\subset U'$, so that $U'$ is non empty and open. We will prove that it is closed in $M_1$. 

Let $x$ be an accumulation point of $U'$. Since we already know that the boundary of $M_1$ belongs to $U'$, we may assume that $x$ belongs to the interior of $M_1$. Then, choose $x_n\in U'$ converging to $x$ and let $x_n':=\Phi(x_n)$. Up to extracting a subsequence, $x_n' \to x'$ for some $x' \in M_2$. By continuity of $\Phi_1$ and $\Phi_2$, we get $\Phi_1(x) = \Phi_2(x')$, in particular $x \in U$. Since $\Phi_2\circ \Phi=\Phi_1$ in $U'$ with $\Phi$ smooth there, we have $D\Phi_2(x_n')D\Phi(x_n)=D\Phi_1(x_n)$. In particular $D\Phi_2(x_n')(T_{x_n'}M_2)=D\Phi_1(x_n)(T_{x_n}M_1)$, and since $D\Phi_i$ is continuous with constant rank, we obtain that (since $x'=\Phi(x)$)
\[ T_{\Phi_2(x')}\Phi_2(M_2)=D\Phi_2(x')(T_{x'}M_2)=D\Phi_1(x)(T_xM_1)=T_{\Phi_1(x)}\Phi_1(M_1).\] 

Next, we show that $x$ is actually in $U'$. For this, let $u=\Phi_1(x)$ and $\mc{V}:=T_u\Phi_1(M_1)\subset L^2(N_{\eps/4})$ and $\mc{V}^\perp$ the orthogonal to $\mc{V}$ for the $L^2(N_{\eps/4})$ scalar product. Denote also by 
$P:L^2(N_{\eps/4})\to \mc{V}$ the orthogonal projection, and let 
\[ 
P\Phi_i : M_i \to \mc{V}.
\]
This is an analytic map, and the differential $DP\Phi_1(x)$ and $DP\Phi_2(x')$ are surjective by construction, as $DP\Phi_i =PD\Phi_i$, and are also injective by the fact that $\Phi_i$ is an embedding. By the local inverse theorem in the analytic category, there is a connected open neighbourhood $W$ of $u$ in $\mc{V}$ and an analytic map $H_i:W\to M_i$, diffeomorphic on its image, such that 
\[
P\Phi_i\circ H_i(w)=w, \quad \forall w\in W.
\]
Thus $\Phi_i\circ H_i(w)=w+ (1-P)\Phi_i\circ H_i(w)$ and if we let $Q_i(w):=(1-P)\Phi_i\circ H_i(w)\in \mc{V}^\perp$, we see that $\Phi_i(M_i)$ is locally the graph of $Q_i|_{W}$. Moreover $Q_i$ is analytic and the graph of $Q_1$ coincides with the graph of $Q_2$ in a non empty open subset $W'\subset W$ since $x$ is an accumulation point of $U'$, thus $Q_1=Q_2$ on $W'$. This implies by analytic continuation that $Q_1=Q_2$ in $W$. In particular there is an open 
neighbourhood $V:=H_1(W)$ of $x$ in $U$, and $\Phi$ is thus well-defined in $V$. Since $\Phi_1\circ H_1=\Phi_2\circ H_2$ on $W$, we also 
get $\Phi=\Phi_2^{-1}\circ \Phi_1=H_2\circ H_1^{-1}$ on $V$, in particular $\Phi$ is analytic on $V$. Hence, $x \in U'$, and thus $U'$ is closed.

Hence, $U'$ is open, closed and non-empty, so that $U' = M_1$. By definition, $\Phi$ is $C^1$ on $U'$, the proof above show that it is actually analytic. The relation $\Phi_2 \circ \Phi = \Phi_1$ proves that $\Phi$ is injective (since $\Phi_1$ is). Notice that the image of $\Phi$ is closed (since $M_1$ is compact), open (since $\Phi$ is a local diffeomorphism) and non-empty (since $M_1$ is non-empty). Hence, $\Phi(M_1) = M_2$, and $\Phi$ is a real-analytic diffeomorphism from $M_1$ to $M_2$.

It remains to prove that $\Phi$ is an isometry. For this, it suffices to observe that $\Phi$ is an isometry when restricted to $U_1\cap M_1$, according to Lemma \ref{determinationhr}. By analytic continuation, $\Phi$ is thus an isometry everywhere. 
\end{proof}

\section{Analyticity of the kernel of the Normal Operator away from the diagonal.}
\label{sec:FBIresults}

This section is dedicated to the proof of Proposition \ref{analyticityofPi0}. The proof relies on an analytic wave front set computation, based on classical results (ellipticity, propagation of singularities,\dots) and a new radial estimates in the analytic category, Proposition \ref{prop:radial_estimate}. The $C^\infty$ analogue of this estimate is now a standard tool in scattering theory (see for instance \cite[Theorem E.42]{DyZw}).

In \S \ref{ssec:FBItranform}, we recall the results from \cite{BJ20} that we will use. In \S \ref{ssec:analytic_wave_front_set}, we discuss the notion of analytic wave front sets, and how one can investigate it using the FBI transform. In \S \ref{ssec:estimates}, we prove the microlocal radial estimates needed for the proof of Proposition \ref{analyticityofPi0}, proof that is given in \S \ref{ssec:analytic_pseudor}.

\subsection{The FBI transform}\label{ssec:FBItranform}

\subsubsection{Basic properties}

Let $\mathcal{M}$ be a closed real-analytic manifold (in the proof of Proposition \ref{analyticityofPi0}, $\mathcal{M}$ will be an extension of the unit tangent bundle $SM_e$ of $M_e$, see \S \ref{ssec:estimates}). For convenience, we will endow $\mathcal{M}$ with a real-analytic Riemannian metric. We then have a notion of Grauert tube of radius $\epsilon>0$ of $\mathcal{M}$, which we denote $(\mathcal{M})_\epsilon$. We endow $T^\ast \mathcal{M}$ with the corresponding real-analytic Kohn--Nirenberg metric $g_{{\rm KN}}=\mathrm{d}x^2+\frac{\mathrm{d}\xi^2}{\cjg \xi\cjd^2}$, and consider the associated Grauert tube $(T^* \mathcal{M})_\epsilon$. This is a conic, complex, pseudo-convex neighbourhood of $T^* \mathcal{M}$. We define on it a Japanese bracket $(T^\ast\mathcal{M})_\epsilon\owns \alpha \mapsto \langle|\alpha|\rangle\in [1,+\infty[$ and the distance $d_{KN}$ associated with the Kohn--Nirenberg metric (see pages 23-24 in \cite{BJ20} for further discussion of these definitions). For a subset $A$ of $\mathcal{M}$, we will write $T^*_{A} \mathcal{M}$ for the restriction of $T^\ast\mathcal{M}$ to $A$.

In order to study the microlocal analytic regularity of distributions on the closed real-analytic manifold $\mathcal{M}$, we introduce an \emph{analytic FBI transform}. We follow the approach exposed in \cite[Chapter 2]{BJ20}, whose main features we recall here. An analytic FBI transform is an operator $T : \mathcal{D}'(\mathcal{M}) \to C^\infty(T^*\mathcal{M})$ with real-analytic kernel $K_T$:
\begin{equation*}
Tu (\alpha) = \int_{\mathcal{M}} K_T(\alpha,y) u(y) \mathrm{d}y, \quad u \in \mathcal{D}'(\mathcal{M}), \quad \alpha = (\alpha_x,\alpha_\xi) \in T^*\mathcal{M}.
\end{equation*}
The kernel $K_T$ depends on an (implicit) small semi-classical parameter $h>0$ and has the following properties: 
\begin{itemize}
\item real-analyticity: $K_T$ has a holomorphic extension to $(T^* \mathcal{M})_\epsilon \times (\mathcal{M})_{\epsilon}$ for some small $\epsilon$ that does not depend on $h$.
\item $K_T$ is negligible when $y$ and $\alpha_x$ are away from each other: for every $\delta > 0$, there are $\epsilon' \in ]0,\epsilon[$ and $ C >0$ such that if $h>0$ is small enough, $(\alpha,y) \in (T^* \mathcal{M})_{\epsilon'} \times (\mathcal{M})_{\epsilon'}$ and the distance between $\alpha_x$ and $y$ is larger than $\delta$ then
\begin{equation*}
\left|K_T(\alpha,y)\right| \leq C \exp\left(- \frac{\left\langle|\alpha|\right\rangle}{Ch}\right).
\end{equation*}
\item local behaviour: for $(\alpha,y) \in (T^* \mathcal{M})_\epsilon \times (\mathcal{M})_\epsilon$, if $\alpha_x$ and $y$ are close to each other, $K_T(\alpha,y)$ is given by
\begin{equation*}
K_T(\alpha,y) \simeq e^{i \frac{\Phi_T(\alpha,y)}{h}} a(\alpha,y),
\end{equation*}
up to an error exponentially decaying in $\jap{|\alpha|}/h$. Here, $\Phi_T$ and $h^{3n/4} a$ are analytic symbols of orders respectively $1$ and $n/4$ on a neighbourhood of uniform size for the Kohn--Nirenberg metric of $\{\alpha_x = y\}$ in $(T^* \mathcal{M})_\epsilon \times (\mathcal{M})_\epsilon$. See \cite[Definition 1.1 and Lemma 1.6]{BJ20} for the definition of an analytic symbols (the class of analytic symbols of order $m$ is called $S^{1,m}$ there).
\end{itemize}
Moreover, the phase $\Phi_T$ satisfies the following properties:
\begin{itemize}
\item if $(\alpha,y) \in T^* \mathcal{M} \times \mathcal{M}$ then the imaginary part of $\Phi_T(\alpha,y)$ is non-negative;
\item $\Phi_T(\alpha,\alpha_x) = 0$ for $\alpha = (\alpha_x,\alpha_\xi) \in T^* \mathcal{M}$; 
\item for $\alpha \in T^* \mathcal{M}$, we have $\mathrm{d}_y \Phi_T(\alpha,\alpha_x) = - \alpha_\xi$;
\item there is $C > 0$, such that for $(\alpha,y) \in T^* \mathcal{M} \times \mathcal{M}$ near $\{\alpha_x = y\}$ we have $\Im \Phi_T(\alpha,y) \geq C^{-1} \jap{\va{\alpha}} d(\alpha_x,y)^2$.
\end{itemize}
In the terminology of \cite[Definition 1.7]{BJ20}, $\Phi_T$ is an admissible phase.

Due to the real-analyticity condition on the kernel $K_T$ of $T$, it is not clear \emph{a priori} that such a transform $T$ exists. However, from \cite[Theorem 6]{BJ20}, there is such a $T$, assuming $h$ is small enough. Moreover, if we endow $T^* \mathcal{M}$ with the volume associated with the canonical symplectic form on $T^* \mathcal{M}$, then we may assume that $T$ is an isometry from $L^2(\mathcal{M})$ into $L^2(T^* \mathcal{M})$. Hence, denoting by $S$ the adjoint of $T$, the composition $ST$ is the identity operator. Notice that $S$ is an operator with real-analytic kernel given for $\alpha$ and $y$ real by
\begin{equation*}
K_S(y,\alpha) = \overline{K_T(\alpha,y)}.
\end{equation*}

Due to the real-analyticity of $K_T$, we see that if $u$ is a distribution (or even a hyperfunction), then $Tu$ is in fact real-analytic, with a holomorphic extension to $(T^* \mathcal{M})_\epsilon$. Hence, if $\Lambda$ is a perturbation of $T^* \mathcal{M}$ in $(T^* \mathcal{M})_\epsilon$, we can define a FBI transform associated with $\Lambda$ by restriction: $T_\Lambda u \coloneqq (Tu)_{|_\Lambda}$. Of course, in order to have interesting properties for $T_\Lambda$, we need to make assumptions on $\Lambda$. We will only consider $\Lambda$'s of the form
\begin{equation}\label{eq:def_Lambda}
\Lambda = \Lambda_G = e^{H_G^{\omega_I}} T^* \mathcal{M}.
\end{equation}
Here, $G$ is a real valued symbol of order at most $1$ in the Kohn--Nirenberg class on $(T^* \mathcal{M})_\epsilon$, and $H_G^{\omega_I}$ is the Hamiltonian vector field of $G$ for the real symplectic form $\omega_I$, which is the imaginary part of the canonical complex symplectic form on the cotangent bundle of the complexification of $\mathcal{M}$. For practical purposes, we will always assume that $G$ is a small symbol of order $1$, so that $\Lambda$ is $C^\infty$ close to $T^* \mathcal{M}$. In particular, the real part $\omega_R$ of the canonical complex symplectic form defines a symplectic form on $\Lambda$.

From the geometric properties of $\Lambda$, there is a symbol $H$ on $\Lambda$, of the same order as $G$ such that $\mathrm{d} H = - \Im \theta_{|\Lambda}$, where $\theta$ denotes the canonical complex Liouville $1$-form on the cotangent bundle of the complexification of $\mathcal{M}$. Letting $\mathrm{d}\alpha$ be the volume form associated with the symplectic form $\omega_R$ on $\Lambda$, we will work with the weighted $L^2$ spaces
\begin{equation*}
L_k^2(\Lambda) \coloneqq L^2\p{\Lambda, \jap{|\alpha|}^{2 k} e^{- \frac{2H}{h}} \mathrm{d}\alpha}, \quad k \in \R,
\end{equation*} 
and their equivalents on $\mathcal{M}$
\begin{equation*}
\mathcal{H}_\Lambda^k \coloneqq \{ u \in \mathcal{O}((\mathcal{M})_\delta)'\ : T_\Lambda u \in L_k^2(\Lambda) \}.
\end{equation*}
Here, $\delta>0$ is fixed, depending on the geometry of $\mathcal{M}$ and $\mathcal{O}((\mathcal{M})_\delta)$ is the closure of bounded holomorphic functions on $(M)_{2\delta}$, in the space of holomorphic functions on $(\mathcal{M})_\delta$ endowed with the $\sup$ norm, and $\mathcal{O}((\mathcal{M})_\delta)'$ its topological dual. According to \cite[Corollary 2.2]{BJ20}, these are complex Hilbert spaces, and according to \cite[Corollary 2.3]{BJ20}, real analytic functions are dense in them. We will also use $\mathcal{H}_\Lambda^\infty = \cap_{k \in \R} \mathcal{H}_\Lambda^k$. 

Later on, we will work with $G$'s of order $\log\langle\xi\rangle$ instead of $\langle\xi\rangle$. In that case, the spaces $\mathcal{H}_\Lambda^k$ only contain distributions, and we can replace $\mathcal{O}((\mathcal{M})_\delta)'$ by $\mathcal{D}'(\mathcal{M})$ in the definition. We will thus avoid completely the discussion of hyperfunctions. However, the reader should keep in mind that as $h\to 0$, the distributions in the space $\mathcal{H}_\Lambda^k$ will potentially be of order $C/h$. The norm on $\mathcal{H}_\Lambda^k$ is given by
\begin{equation}\label{eq:norme_Lambda}
\n{u}_{\mathcal{H}_\Lambda^k}^2 = \int_{\Lambda} \va{T_\Lambda u}^2 \jap{|\alpha|}^{2k} e^{- \frac{2H}{h}} \mathrm{d}\alpha
\end{equation}
for $u \in \mathcal{D}'(\mathcal{M})$.

By restricting the kernel of $S$ to $\mathcal{M} \times \Lambda$, we get an operator $S_\Lambda$ that turns functions on $\Lambda$ growing at a subexponential rate into elements of $\mathcal{O}((\mathcal{M})_\delta)'$. Moreover, the composition $S_\Lambda T_\Lambda$ is well-defined, and is the identity operator \cite[Lemma 2.7]{BJ20}. The structure of the operator $\Pi_\Lambda \coloneqq T_\Lambda S_\Lambda$ is given in \cite[Lemmata 2.9 and 2.10]{BJ20}. In particular \cite[Proposition 2.4]{BJ20}, $\Pi_\Lambda$ is bounded on $L_k^2(\Lambda)$ for every $k \in \R$ . 

The motivation for working with such a $\Lambda$ is the following central result (this is a particular case of \cite[Theorem 7]{BJ20}).

\begin{proposition}[Multiplication formula]\label{theorem:multiplication_formula}
Let $P$ be an analytic semi-classical differential operator of order $m \in \mathbb{N}$ on $\mathcal{M}$. Assume that $G$ in the definition of $\Lambda$ is small enough. Then there is a symbol $q$ of order $m$ on $\Lambda$ such that for every $u \in \mathcal{H}_\Lambda^{m/2}$ we have
\begin{equation*}
\int_{\Lambda} T_\Lambda (P u) \overline{T_\Lambda u} e^{- \frac{2H}{h}} \mathrm{d}\alpha = \int_\Lambda q \va{T_\Lambda u}^2 e^{- \frac{2H}{h}} \mathrm{d}\alpha + \mathcal{O}(h^\infty).
\end{equation*}
More precisely the remainder is smaller than $C_N h^N \|u\|_{\mathcal{H}^{(m-N)/2}_\Lambda}^2$ for $N\geq 0$. 

Additionally, up to a $\mathcal{O}(h)$ in the space of Kohn--Nirenberg symbols of order $m-1$ on $\Lambda$, the symbol $q$ is the restriction to $\Lambda$ of the holomorphic continuation of the principal symbol of $P$.
\end{proposition}

In order to prove Proposition \ref{analyticityofPi0}, we would like to apply this formula to the generator of the geodesic flow on $SM_e$. However, once $SM_e$ is embedded into a closed real-analytic manifold, the generator of the geodesic flow is only analytic near $SM_e$, so that we need a local version of the multiplication formula. We explain now how this difficulty can be overcome.

\subsubsection{Pseudo-locality properties}\label{sssec:local}

The pseudo-locality property of the FBI transform is the first thing to notice.

\begin{lemma}\label{lemma:pseudo_locality}
Let $K_0$ and $K_1$ be compact subsets of $\mathcal{M}$ that do not intersect. There exist $C_0> 0$ and a conic neighbourhood $V$ of $T^*_{K_1} \mathcal{M}$ in $(T^* \mathcal{M})_\epsilon$ such that, for every $N \in \R$, there is a constant $C$ such that for every $u \in H^{-N}(\mathcal{M})$ supported in $K_0$ and $\alpha \in V$, we have
\begin{equation*}
\va{Tu(\alpha)} \leq C \n{u}_{H^{-N}} \exp\p{- \frac{\jap{\va{\alpha}}}{C_0h}}.
\end{equation*}
\end{lemma}

\begin{proof}
This is an immediate consequence of the assumption that the kernel $K_T$ of $T$ is small away from the diagonal (the second assumption on $K_T$).
\end{proof}

This consequence of Lemma \ref{lemma:pseudo_locality} will be used several times.

\begin{lemma}\label{lemma:pseudo_local_space}
Let $K_0$ be a compact subset of $\mathcal{M}$. Let $K_1$ be a compact subset of $\mathcal{M}$ that does not intersect $K_0$. Assume that $G$ is a small enough symbol of order $1$ on $(T^* \mathcal{M})_\epsilon$ and let $\Lambda = \Lambda_G$. Let $N,k \in \R$. There is a constant $C > 0$ such that for every $u \in H^{-N}(\mathcal{M})$ supported in $K_0$ we have
\begin{equation*}
\int_{\Lambda \cap (T^*_{K_1} \mathcal{M})_\epsilon} \jap{\va{\alpha}}^{2k} \va{Tu(\alpha)}^2 e^{- \frac{2H(\alpha)}{h}} \mathrm{d}\alpha \leq C \n{u}_{H^{-N}}^2 \exp\p{- \frac{1}{Ch}}.
\end{equation*}
\end{lemma}

We end this section with the adaptation of Proposition \ref{theorem:multiplication_formula} to differential operators that are analytic only on a part of $\mathcal{M}$. We use here the norm introduced in \eqref{eq:norme_Lambda}.

\begin{proposition}\label{theorem:local_multiplication_formula}
Let $P$ be a semi-classical differential operator of order $m$ on $\mathcal{M}$. Let $K_0$ be a compact subset of $\mathcal{M}$. Assume that $P$ has real-analytic coefficients on a neighbourhood of $K_0$. Assume that $G$ is a small enough symbol of order $1$ on $(T^* \mathcal{M})_\epsilon$ and let $\Lambda = \Lambda_G$. There exists a Kohn--Nirenberg symbol $q$ of order $m$ on $\Lambda$ such that for every $N,m_1,m_2 \in \R$ such that $m = m_1 + m_2 -1$, there exists a constant $C > 0$ such that for every $u \in H^{-N}(\mc{M}) \cap \mathcal{H}_\Lambda^\infty$ supported in $K_0$ we have
\begin{equation*}
\va{\int_\Lambda T_\Lambda ( P u) \overline{T_\Lambda u} e^{- \frac{2H}{h}} \mathrm{d}\alpha - \int_{\Lambda} q \va{Tu}^2 e^{- \frac{2H}{h}}\mathrm{d}\alpha} \leq C \exp\p{- \frac{1}{Ch}} \n{u}_{H^{-N}}^2 + Ch \n{u}_{\mathcal{H}_\Lambda^{m_1}} \n{u}_{\mathcal{H}_\Lambda^{m_2}}.
\end{equation*}
Moreover, near $T^*_{K_0} \mathcal{M}$, the symbol $q$ is the restriction of the principal symbol of $P$ to $\Lambda$.
\end{proposition}

In order to prove Proposition \ref{theorem:local_multiplication_formula}, let $\chi \in C^\infty((\mathcal{M})_\epsilon)$ be such that $\chi \equiv 1$ on a neighbourhood of $K_0$, and $P$ has a holomorphic extension on a neighbourhood of the support of $\chi$. We identify $\chi$ with the function $\alpha \mapsto \chi(\alpha_x)$ on $(T^* \mathcal{M})_\epsilon$, and with the corresponding multiplication operator. Proposition \ref{theorem:local_multiplication_formula} will be deduced from the two following lemmas.

\begin{lemma}\label{lemma:to_localize}
Under the assumption of Proposition \ref{theorem:local_multiplication_formula}, we have
\begin{equation*}
T_\Lambda Pu(\alpha)=\chi T_\Lambda PS_\Lambda \chi T_\Lambda u(\alpha) + \mathcal{O}\p{\exp\p{- \frac{\jap{\va{\alpha}}}{Ch}} \n{u}_{H^{-N}}},
\end{equation*}
for $\alpha \in \Lambda$.
\end{lemma}

\begin{lemma}\label{lemma:toeplitz}
Under the assumption of Proposition \ref{theorem:local_multiplication_formula}, let $B_\Lambda$ denote the orthogonal projector from $L_0^2(\Lambda)$ to its closed subspace $T_\Lambda \mathcal{H}_\Lambda^0$. Then there exists a Kohn--Nirenberg symbol $\tilde{q}$ of order $m$ on $\Lambda$ such that
\begin{equation*}
B_\Lambda \chi T_\Lambda P S_\Lambda \chi B_\Lambda = B_\Lambda \tilde{q} B_\Lambda + R,
\end{equation*}
where $R$ is $\mathcal{O}(h^\infty)$ as an operator from $L^2_{k}(\Lambda)$ to $L^2_{-k}(\Lambda)$ for every $k \in \R$. Moreover, $\tilde{q} = \chi^2 q \textup{ mod } h S_{KN}^{m-1}(\Lambda)$ where $q$ is the restriction of the principal symbol of $P$ to $\Lambda$ near $T^*_{K_0} \mathcal{M}$.
\end{lemma}

Before proving Lemmas \ref{lemma:to_localize} and \ref{lemma:toeplitz}, let us explain how they can be used to prove Proposition \ref{theorem:local_multiplication_formula}.

\begin{proof}[Proof of Proposition \ref{theorem:local_multiplication_formula}]
Applying Lemma \ref{lemma:to_localize} and then Lemma \ref{lemma:toeplitz}, we find that:
\begin{equation*}
\begin{split}
& \int_\Lambda T_\Lambda P u. \overline{T_\Lambda u} e^{- \frac{2H}{h}} \mathrm{d}\alpha = \int_\Lambda \chi T_\Lambda P S_\Lambda \chi T_\Lambda u. \overline{T_\Lambda u} e^{- \frac{2H}{h}} \mathrm{d}\alpha + \mathcal{O}\p{\exp\p{- \frac{1}{Ch}}\n{u}_{H^{-N}}^2} \\
       & \quad \quad \qquad = \int_\Lambda \chi T_\Lambda P S_\Lambda \chi B_\Lambda T_\Lambda u. \overline{B_\Lambda T_\Lambda u} e^{- \frac{2H}{h}} \mathrm{d}\alpha + \mathcal{O}\p{\exp\p{- \frac{1}{Ch}}\n{u}_{H^{-N}}^2} \\
       & \quad \quad \qquad = \int_\Lambda B_\Lambda \chi T_\Lambda P S_\Lambda \chi B_\Lambda T_\Lambda u. \overline{T_\Lambda u} e^{- \frac{2H}{h}} \mathrm{d}\alpha + \mathcal{O}\p{\exp\p{- \frac{1}{Ch}}\n{u}_{H^{-N}}^2} \\
       & \quad \quad \qquad = \int_\Lambda B_\Lambda \tilde{q} B_\Lambda T_\Lambda u. \overline{T_\Lambda u} e^{- \frac{2H}{h}} \mathrm{d}\alpha + \mathcal{O}\p{\exp\p{- \frac{1}{Ch}}\n{u}_{H^{-N}}^2 + h^\infty \n{u}_{\mathcal{H}_\Lambda^{m_1}} \n{u}_{\mathcal{H}_\Lambda^{m_2}}} \\
       & \quad \quad \qquad = \int_\Lambda \tilde{q}  \va{T_\Lambda u}^2 e^{- \frac{2H}{h}} \mathrm{d}\alpha + \mathcal{O}\p{\exp\p{- \frac{1}{Ch}}\n{u}_{H^{-N}}^2 + h^\infty \n{u}_{\mathcal{H}_\Lambda^{m_1}} \n{u}_{\mathcal{H}_\Lambda^{m_2}}} \\
      & \quad \quad \qquad = \int_\Lambda q  \va{T_\Lambda u}^2 e^{- \frac{2H}{h}} \mathrm{d}\alpha + \mathcal{O}\p{\exp\p{- \frac{1}{Ch}}\n{u}_{H^{-N}}^2 + h \n{u}_{\mathcal{H}_\Lambda^{m_1}} \n{u}_{\mathcal{H}_\Lambda^{m_2}}}. \\ 
\end{split}
\end{equation*}
On the first line, we have not used $u\in\mathcal{H}^\infty_\Lambda$, but rather that $T_\Lambda u(\alpha)$ is at most of the size $\mathcal{O}(e^{\delta\jap{\va{\alpha}}/h}\n{u}_{H^{-N}})$ on $\Lambda$, with $\delta$ that can be made arbitrarily small by taking $G$ small enough (as follows from a direct inspection of $K_T$).
\end{proof}

\begin{proof}[Proof of Lemma \ref{lemma:to_localize}]
We start by applying Lemma \ref{lemma:pseudo_locality} to find that, for $u \in H^{-N}(\mc{M}) \cap \mathcal{H}_\Lambda^\infty$ supported in $K_0$,
\begin{equation*}
\begin{split}
T_\Lambda Pu(\alpha) & = \chi T_\Lambda Pu(\alpha)+ \mathcal{O}\p{\exp\p{- \frac{\jap{\va{\alpha}}}{Ch}} \n{u}_{H^{-N}}} \\
            & = \chi T_\Lambda PS_\Lambda T_\Lambda u(\alpha) + \mathcal{O}\p{\exp\p{- \frac{\jap{\va{\alpha}}}{Ch}} \n{u}_{H^{-N}}} \\
            & = \chi T_\Lambda PS_\Lambda \chi T_\Lambda u(\alpha) + \chi T_\Lambda PS_\Lambda (1 - \chi) T_\Lambda u(\alpha) + \mathcal{O}\p{\exp\p{- \frac{\jap{\va{\alpha}}}{Ch}} \n{u}_{H^{-N}}}. \\
\end{split}
\end{equation*}
Now, we claim that $\chi T_\Lambda PS_\Lambda (1 - \chi) T_\Lambda u(\alpha)$ can in fact be put in the error term. The Schwartz kernel of the operator $\chi T_\Lambda PS_\Lambda (1 - \chi)$ is given by
\begin{equation*}
\chi(\alpha_x)(1 - \chi(\beta_x)) \int_{\mathcal{M}} K_T(\alpha,y) P_y(K_S(y,\beta)) \mathrm{d}y.
\end{equation*}
We can split the integral here into the $y$'s that are close to $\alpha_x$ and those that are not. If $y$ is away from $\alpha_x$ then by assumption the kernel $K_T(\alpha,y)$ is $\mathcal{O}\p{\exp(- \jap{\va{\alpha}}/Ch)}$. On the other hand $K_S(y,\beta)$ is always $\mathcal{O}(\exp( \delta \jap{\beta}/h))$ where $\delta >0$ can be chosen arbitrarily small by imposing that $\beta$ remains close to real (which will be verified if $G$ is small enough). Hence, the $y$'s away from $\alpha_x$ contribute to the kernel of $\chi T_\Lambda PS_\Lambda (1 - \chi)$ by an
\begin{equation}\label{eq:size_kernel}
\mathcal{O}\p{\exp\p{- \frac{\jap{\va{\alpha}}}{C h} + \frac{\delta \jap{\va{\beta}}}{h}}}.
\end{equation}
When $y$ is close to $\alpha_x$, then $P_y(K_S(y,\beta))$ is analytic in $y$, with a holomorphic extension bounded by $\mathcal{O}(\exp( \frac{\delta \jap{\beta}}{h}))$ on a neighbourhood of $\alpha_x$ of fixed size. Hence, we can apply the non-stationary phase method \cite[Proposition 1.1]{BJ20}, as in the proof of \cite[Lemma 2.4]{BJ20} for instance (see the estimate on $T^D u$ defined by (2.25) there), to find that the $y$'s close to $\alpha$ contribute to the kernel of $\chi T_\Lambda PS_\Lambda (1 - \chi)$ by a term of the same size as \eqref{eq:size_kernel}. The kernel itself is then of size at most \eqref{eq:size_kernel}. Using Lemma \ref{lemma:pseudo_locality} to estimate the size of $Tu(\beta)$ when $\chi(\beta) \neq 1$, and choosing $\delta > 0$ small enough, we find indeed that $\chi T_\Lambda PS_\Lambda (1 - \chi) T_\Lambda u(\alpha)$ is an $\mathcal{O}\p{\exp\p{- \frac{\jap{\va{\alpha}}}{Ch}} \n{u}_{H^{-N}}}$, so that we have
\begin{equation}\label{eq:kernel_near_K0}
T_\Lambda Pu(\alpha) = \chi T_\Lambda PS_\Lambda \chi T_\Lambda u(\alpha) + \mathcal{O}\p{\exp\p{- \frac{\jap{\va{\alpha}}}{Ch}} \n{u}_{H^{-N}}}.\qedhere
\end{equation}
\end{proof}

\begin{proof}[Proof of Lemma \ref{lemma:toeplitz}]
Lemma \ref{lemma:toeplitz} is the equivalent of \cite[Proposition 2.10]{BJ20}. The proof of the latter can be rewritten here without changing a word of it, once we prove that the operator $\chi T_\Lambda P S_\Lambda \chi$ satisfies \cite[Lemma 2.9 and 2.10]{BJ20} (see also \cite[Remark 2.15]{BJ20}). Thus, let us explain how to prove that the operator $\chi T_\Lambda P S_\Lambda \chi$ satisfies \cite[Lemma 2.9 and 2.10]{BJ20}.

The kernel of the operator $\chi T_\Lambda PS_\Lambda \chi$ is given for $\alpha,\beta \in (T^* \mathcal{M})_\epsilon$ by
\begin{equation*}
\chi T_\Lambda PS_\Lambda \chi(\alpha,\beta) = \chi(\alpha_x)\chi(\beta_x) \int_{\mathcal{M}} K_T(\alpha,y) P_y(K_T(y,\beta)) \mathrm{d}y.
\end{equation*}
Since the kernel $K_T(\alpha,y)$ and $K_S(y,\beta)$ are negligible when $y$ is away from $\alpha_x$ and $\beta_x$, the $y$'s in this integral that are away from the support of $\chi$ only contributes by a negligible term: if $U$ is a neighbourhood of $\textup{supp } \chi \cap \mathcal{M}$, then there is $\epsilon'> 0$ such that
\begin{equation}\label{eq:avec_U}
\chi T_\Lambda PS_\Lambda \chi(\alpha,\beta) = \chi(\alpha_x)\chi(\beta_x) \int_{U} K_T(\alpha,y) P_y(K_T(y,\beta)) \mathrm{d}y + \mathcal{O}\p{\exp\p{- \frac{\jap{\va{\alpha}} + \jap{\va{\beta}}}{Ch}}},
\end{equation}
for $\alpha,\beta \in (T^* \mathcal{M})_{\epsilon'}$. The remaining integral over $U$ only involves real-analytic functions that are exponentially small on the boundary of $U$, and can consequently be dealt with following exactly the same steps as in the proof of Lemmas 2.9 and 2.10 in \cite{BJ20} in the analytic case pp.110--111 (here we do not need the much more complicated proof required for the Gevrey case). For the convenience of the reader, we recall here the main lines of the argument.

The equivalent of \cite[Lemma 2.9]{BJ20} is that the kernel of $\chi T_\Lambda PS_\Lambda \chi$ is negligible away from the diagonal: for every $\eta > 0$, there are $C,\epsilon' > 0$ such that if $\alpha,\beta \in (T^* \mathcal{M})_{\epsilon'}$ are such that $d_{KN}(\alpha,\beta) \geq \eta/2$ then
\begin{equation*}
\chi T_\Lambda PS_\Lambda \chi(\alpha,\beta) = \mathcal{O}\p{\exp\p{- \frac{\jap{\va{\alpha}} + \jap{\va{\beta}}}{Ch}}}.
\end{equation*}
The proof of this fact is very similar to the proof of Proposition \ref{prop:equivalence_wave_front_set} given in Appendix \ref{appendix:wave_front_set}. Notice that we only need to consider the case of $\alpha_x$ and $\beta_x$ in the support of $\chi$. We will take $\epsilon' \ll \eta \ll 1$ and let $\tilde{\alpha}_x$ and $\tilde{\beta}_x$ denote elements of $\mathcal{M}$ that are at distance at most $\epsilon'$ of $\alpha_x$ and $\beta_x$ respectively. If the distance between $\alpha_x$ and $\beta_x$ is larger than $\eta/4$, we write
\begin{equation*}
\begin{split}
& \int_{U} K_T(\alpha,y) P_y(K_T(y,\beta)) \mathrm{d}y  \\ & \qquad \quad= \p{\int_{D(\tilde{\alpha}_x,\eta/100)} + \int_{D(\tilde{\beta}_x,\eta/100)} + \int_{U \setminus (D(\tilde{\alpha}_x,\eta/100) \cup D(\tilde{\beta}_x,\eta/100))}} K_T(\alpha,y) P_y(K_T(y,\beta)) \mathrm{d}y.
\end{split}
\end{equation*}
The third integral is negligible because $K_T(\alpha,y)$ and $K_S(y,\beta)$ are negligible when $y$ is away from $\alpha_x$ and $\beta_x$ (this is the second assumption we made on $K_T$). The other two integral are dealt with similarly. For the first one for instance, we use the local behaviour of $K_T$ to rewrite it, up to negligible terms, as
\begin{equation}\label{eq:truc_complexe}
\int_{D(\tilde{\alpha}_x,\eta/100)} e^{i \frac{\Phi_T(\alpha,y)}{h}} a(\alpha,y) P_y(K_T(y,\beta)) \mathrm{d}y.
\end{equation}
Here, since $\alpha_x$ is away from $\beta_x$, the function $y \mapsto P_y(K_T(y,\beta))$ is analytic with a holomorphic extension bounded by $\mathcal{O}(\exp(- \jap{\va{\beta}}/Ch))$ on a neighbourhood of fixed size of the domain of integration. Moreover, for $\eta$ and $\epsilon'$ small enough the phase $\Phi_T$ is non-stationary (as $\mathrm{d}_y \Phi_T(\alpha,\alpha_x) = - \alpha_\xi$) and has positive imaginary part on the boundary of the domain of integration. We can consequently use the non-stationary phase method \cite[Proposition 1.5]{BJ20} to see that the integral \eqref{eq:truc_complexe} is indeed negligible.

If $\alpha_x$ and $\beta_x$ are close, but $\alpha_\xi$ and $\beta_\xi$ are away from each other, we can use local coordinates and rewrite up to negligible (exponentially decaying) terms
\begin{equation}\label{eq:le_vrai_truc}
\int_{U} K_T(\alpha,y) P_y(K_T(y,\beta)) \mathrm{d}y \simeq \int_{D(\tilde{\alpha}_x,\eta/100)} e^{i \frac{\Phi_T(\alpha,y) + \Phi_S(y,\beta)}{h}} a(\alpha,y) \tilde{b}(y,\beta) \mathrm{d}y,
\end{equation}
where $\tilde{b}$ is an analytic symbol. Since $\mathrm{d}_y \Phi_T(\alpha,\alpha_x) = - \alpha_\xi$ and $\mathrm{d}_y \Phi_S(\beta_x,\beta) = \beta_\xi$, we see that when $\alpha_\xi$ and $\beta_\xi$ are away from each other, the phase in \eqref{eq:le_vrai_truc} is non-stationary again, and the non-stationary phase method gives that \eqref{eq:le_vrai_truc} is negligible (see the proof of Proposition \ref{prop:equivalence_wave_front_set} in Appendix \ref{appendix:wave_front_set} for a similar argument).

The equivalent of \cite[Lemma 2.10]{BJ20} in our context is a description of the structure of the kernel of $\chi T_\Lambda P S_\Lambda \chi$ near the diagonal: there are small $\eta > 0$ and $\epsilon' > 0$, such that for $\alpha,\beta \in (T^* \mathcal{M})_{\epsilon'}$ with $d_{KN}(\alpha,\beta) < \eta$ we have, up to exponentially decaying terms
\begin{equation}\label{eq:local_structure}
\chi T_\Lambda PS_\Lambda \chi(\alpha,\beta) \simeq e^{i \frac{\Phi_{TS}(\alpha,\beta)}{h}} e(\alpha,\beta),
\end{equation}
where $\Phi_{TS}(\alpha,\beta)$ is the critical value of $y \mapsto \Phi_T(\alpha,y) + \Phi_S(y,\beta)$ and $e$ is a Kohn--Nirenberg symbol supported near the diagonal. Moreover, $e$ is given at first order on the diagonal by
\begin{equation*}
e(\alpha,\alpha) = \frac{1}{(2 \pi h)^n} \chi(\alpha)^2 p(\alpha) + \mathcal{O}\p{h^{-n+1} \jap{\va{\alpha}}^{m-1}},
\end{equation*}
where $p$ is the (holomorphic extension of the) principal symbol of $P$. In order to prove \eqref{eq:local_structure}, we just notice that the only relevant term in \eqref{eq:avec_U} is the integral over $U$, and since $\alpha$ and $\beta$ are close to each other, this integral is given up to negligible terms by \eqref{eq:le_vrai_truc}. Then, we observe that when $\alpha = \beta$ the phase $y \mapsto \Phi_T(\alpha,y) + \Phi_S(y,\beta)$ has a non-degenerate critical point at $y = \alpha_x = \beta_x$. Moreover, provided $\epsilon' \ll \eta \ll 1$, the phase is positive on the boundary of the domain of integration. We can consequently apply the holomorphic stationary phase method \cite[Poposition 1.6]{BJ20} (see also \cite[Théorème 2.8 and Remarque 2.10]{SjAsterisque}) to get \eqref{eq:local_structure}.

\end{proof}

\subsection{Some properties of the analytic wavefront set}\label{ssec:analytic_wave_front_set}

\subsubsection{General facts about analytic wave front set}

We will use the notion of analytic wave front set $\WF_a(u)$ of a distribution $u$ on a real-analytic manifold. See for instance \cite[\S 8.4 \& 8.5]{HoeI}, \cite[\S 6]{SjAsterisque} or \cite{HitrikSjostrand} for the definition and basic properties of this notion, the definition from \cite{SjAsterisque} is recalled in Appendix \ref{appendix:wave_front_set}. These references define differently the analytic wave front set, but all the classical definitions of the analytic wave front set coincide according to a result of Bony \cite{Bony}. Notice also that these references deal mainly with the case of distributions on open subsets of $\mathbb{R}^n$, but the case of distributions on a real-analytic manifold follows immediately since the notion is invariant by real-analytic change of coordinates.

Let us start by recalling that if $u$ is a distribution on a real-analytic manifold $\mathcal{M}$ then $\WF_a(u)$ is a closed conic subset of $T^* \mathcal{M} \setminus \set{0}$ whose projection to $\mathcal{M}$ is the analytic singular support of $u$, that is the set of points $x \in \mathcal{M}$ such that $u$ is not analytic on a neighbourhood of $x$ (see for instance  \cite[Theorem 8.4.5]{HoeI} or \cite[Theorem 6.3]{SjAsterisque}).

In this paper, we will use the following characterization of the analytic wave front set in term of the FBI transform. There is no surprise in the proof of this theorem, that can be found in Appendix \ref{appendix:wave_front_set}. We use here the same notation as in \S \ref{ssec:FBItranform}, in particular $\mathcal{M}$ is a closed real-analytic manifold, and $T$ an analytic FBI transform on $\mathcal{M}$.

\begin{proposition}\label{prop:equivalence_wave_front_set}
Let $u \in \mathcal{D}'(\mathcal{M})$ be independent of $h$. Let $\alpha_0 \in T^* \mathcal{M} \setminus \{0\}$. Then the following assertions are equivalent:
\begin{enumerate}[label=(\roman*)]
\item $\alpha_0$ does not belong to $\WF_a(u)$;
\item there is neighbourhood $\Omega \subseteq T^* \mathcal{M}$ of $\alpha_0$ and a constant $C > 0$ such that for $h$ small enough and $\alpha \in \Omega$ we have
\begin{equation}\label{eq:weak_wave_front_set}
\va{Tu(\alpha)} \leq C \exp\p{- \frac{1}{Ch}}.
\end{equation}
\item there is conic neighbourhood $\Omega \subseteq T^* \mathcal{M}$ of $\alpha_0$ and a constant $C > 0$ such that for $h$ small enough and $\alpha \in \Omega$ large enough we have
\begin{equation}\label{eq:strong_wave_front_set}
\va{Tu(\alpha)} \leq C \exp\p{- \frac{\jap{\alpha}}{Ch}}.
\end{equation}
\end{enumerate}
\end{proposition}

\begin{remark}
One could add a fourth equivalent property in Proposition \ref{prop:equivalence_wave_front_set} by letting $\alpha$ tends to $+ \infty$ while $h$ is fixed. Using this characterization of the analytic wave front set, one could prove without much trouble the basic properties of the analytic wave front set (behaviour under pull-back and push-forward for instance). Some other standard properties of the analytic wave front set, such as the elliptic estimates and propagation of singularities may also be proved using the FBI transform, in the spirit of the proof of Propositions \ref{prop:radial_estimate} and \ref{prop:radial_estimate_bis} below.
\end{remark}

We will also need the following result, which follows easily by using Sjöstrand's definition \cite[Definition 6.1]{SjAsterisque}. See Appendix \ref{appendix:wave_front_set} for a proof.

\begin{lemma}\label{lemma:deux_variables}
Let $u \in \mathcal{D}'(\mathcal{M} \times \mathcal{M})$. Let $(x,y) \in \mathcal{M} \times \mathcal{M}$. Assume that $u$ is $C^\infty$ near $(x,y)$ and that $u$ is analytic in the first variable, uniformly in the second variable, on a neighbourhood of $(x,y)$. Then $\WF_a(u) \cap T_x^* \mathcal{M} \times T_y^* \mathcal{M} \subseteq \{0\} \times T_y^* \mathcal{M}$.
\end{lemma}

\subsubsection{Lagrangian deformation and analytic regularity}\label{sssec:space_regularity}

We explain now how the spaces $\mathcal{H}_\Lambda^k$ from \S \ref{ssec:FBItranform} associated with a complex Lagrangian $\Lambda$ are related to regularity issues. In order to do so, we will use a Lagrangian $\Lambda$ defined by \eqref{eq:def_Lambda} using a function $G$ which is a symbol of logarithmic order, meaning that in local coordinates $\tilde{x} = x + iy$ and $\tilde{\xi} = \xi + i \eta$, for every $\gamma,\delta,\rho,\sigma \in \mathbb{N}^{n}$ there is a constant $C_{\gamma,\delta,\rho,\sigma}$ such that
\begin{equation*}
\va{\partial_{x}^\gamma \partial_{y}^\delta \partial_{\xi}^{\rho} \partial_{\eta}^{\sigma} G(\tilde{x},\tilde{\xi})} \leq C_{\gamma,\delta,\rho,\sigma} (1+ \va{\xi})^{- \va{\rho} - \va{\sigma}} \log \p{ 2 + \va{\xi}}
\end{equation*}
for $\tilde{x}$ in a coordinate patch for $(\mathcal{M})_\epsilon$ and $\va{\eta} \leq \epsilon (1+ \va{\xi})$.

In particular, $G$ will be of order $1$, so that the theory from \cite{BJ20} will apply, but since $\Lambda$ is now much closer to $T^* \mathcal{M}$, the spaces $\mathcal{H}_\Lambda^k$ will contain $C^\infty(\mathcal{M})$ (see Lemma \ref{lemma:space_cinfini}). However, $G$ will not depend on $h$, so that a distribution that belongs to $\mathcal{H}_\Lambda^0$ with some uniform estimates in $h$ must have its analytic wave front set localized outside of the set of negative ellipticity of $G$ (see Lemma \ref{lemma:hlambda_to_wfs}). To put it shortly, we are doing $C^\infty$ microlocal analysis from the classical point of view, but real-analytic microlocal analysis in the semiclassical perspective.

Notice that an advantage of working with a symbol of logarithmic order $G$ is that the assumption that $G$ is small enough as a symbol of order $1$ (needed to apply the results from \cite{BJ20}) can easily be satisfied by multiplying $G$ by a function that vanishes on a large neighbourhood of $0$ and takes value $1$ near infinity. As we do not care about the value of $G$ near $0$ in the applications below, this is particularly practical for us.

We start with an estimate which will be useful to discard the part of the phase space that is not of interest for our purpose.

\begin{lemma}\label{lemma:distribution_space}
Let $G$ be a symbol of logarithmic order on $(T^* \mathcal{M})_\epsilon$. Assume that $G$ is small enough as a symbol of order $1$. Let $\Lambda = \Lambda_G$ be the associated Lagrangian defined by \eqref{eq:def_Lambda}. Let $\Omega \subseteq T^* \mathcal{M}$ be such that there is $\delta > 0$ such that for $\alpha \in \Omega$ such that $\jap{\alpha} \geq \delta^{-1}$, we have:
\begin{equation}\label{assumption_lemma:distribution_space}
G(\alpha) \geq \delta \log \jap{\alpha}.
\end{equation}
Write $\Omega' = e^{H_G^{\omega_I}} \Omega \subseteq \Lambda$. There is a constant $C_0 > 0$ such that, for every $N,k \in \R$, there is a constant $C > 0$ such that for all $h$ small enough and every $u \in H^{-N}(\mathcal{M})$ we have 
\begin{equation}\label{eq:controle_distribution}
\int_{\Omega'} \va{T_\Lambda u}^2 e^{- \frac{2H}{h}} \jap{\va{\alpha}}^{2k} \mathrm{d}\alpha \leq C e^{\frac{C_0}{h}} \n{u}_{H^{-N}}^2.
\end{equation}
\end{lemma}
Here the constant $C_0$ can be made arbitrarily small taking $G$ arbitrarily small, but we will not need this. Actually, this growth only comes from small frequencies.

\begin{proof}
We can assume that $u$ is real-analytic, as the general result then follows by an approximation argument. Indeed, using for instance the eigenfunctions of the Lapalce operator associated to a real-analytic metric, one can produce a sequence $(u_n)_{n \geq 0}$ of real-analytic functions converging to $u$ in $H^{-N}(\mathcal{M})$. Then if each $u_n$ satisfies \eqref{eq:controle_distribution} (with constants $C$ and $C_0$ that does not depend on $u_n$), it follows from Fatou's lemma that $u$ also satisfies \eqref{eq:controle_distribution} since $(T_\Lambda u_n)_{n \geq 0}$ converges pointwise to $T_\Lambda u$. Without loss of generality, we may also assume that $N$ is a positive integer. 

Denote by $K_{TS}$ the kernel of $TS$ and its holomorphic extension. Recall that $T_{\Lambda}u=(Tu)_{|\Lambda}$ where $Tu$ is considered as a function on $(T^*\mc{M})_\eps$ by holomorphic extension. Then, since $u$ is real-analytic, $Tu$ decays exponentially fast, and we can use the formula
\begin{equation}\label{eq:formule_avec_noyau}
T_\Lambda u(\alpha) = \int_{T^* \mathcal{M}} K_{TS}(\alpha,\beta) Tu(\beta) \mathrm{d}\beta \quad \textup{ for } \alpha \in \Lambda
\end{equation}
to bound $T_\Lambda u$. The formula \eqref{eq:formule_avec_noyau} follows from the fact that $ST$ is the identity operator.

According to Lemma 2.9 in \cite{BJ20}, given $\eta>0$, if $G$ is small enough, and $d_{KN}(\alpha,\beta) > \eta$, with $\alpha\in\Lambda$ and $\beta\in T^\ast \mathcal{M}$, the kernel $K_{TS}(\alpha,\beta)$ is exponentially small in $(\jap{\va{\alpha}} + \jap{\va{\beta}})/h$. In particular, 
\[
T_\Lambda u(\alpha) = \int_{\beta\in T^* \mathcal{M},\ d_{KN}(\alpha,\beta)\leq \eta} K_{TS}(\alpha,\beta) Tu(\beta) \mathrm{d}\beta + \mathcal{O}(e^{-\langle|\alpha|\rangle/Ch} \| u\|_{H^{-N} } ). 
\]
Here we used the fact that $\mathcal{H}^k_{T^\ast M}$ is the semi-classical Sobolev space of order $k$ (see Corollary 2.4 in \cite{BJ20}). 

Let us turn now to the points close to $\alpha$. Here we can use Lemma 2.10 of \cite{BJ20}. For $\eta>0$ small enough and $d_{KN}(\alpha,\beta)<\eta$, 
\begin{equation}\label{eq:noyau_TS}
K_{TS}(\alpha,\beta) = e^{i \frac{\Phi_{TS}(\alpha,\beta)}{h}} e(\alpha,\beta) + \mathcal{O}\p{\exp \p{ - \frac{ \langle|\alpha|\rangle }{h} }},
\end{equation}
where $h^n e$ is a symbol of order $0$ and the phase $\Phi_{TS}(\alpha,\beta)$ is the critical value of $y \mapsto \Phi_T(\alpha,y)+\Phi_S(\beta,y)$. This formula is only stated for $\alpha,\beta \in \Lambda$ in \cite[Lemma 2.10]{BJ20}, but since we can rely here on the proof given in the analytic case, the formula actually holds as soon as $\alpha$ and $\beta$ are close enough to the diagonal of $T^* \mathcal{M}$ (for the same reason, $e$ is actually analytic near the diagonal). This is explained at the start of the proof (p.111) and also in the remark before figure $1$ therein.

Now using the same argument as before to deal with the remainder, we can concentrate on
\[
I(\alpha) :=\int_{d_{KN}(\alpha,\beta)<\eta} e^{i\frac{\Phi_{TS}(\alpha,\beta)}{h}} e(\alpha,\beta) T u(\beta) \mathrm{d}\beta
\]
in order to bound $T_\Lambda u(\alpha)$. Choose $R>\delta^{-1}$ large. For $|\alpha|<R$, the integral $I(\alpha)$ is controlled by
\begin{equation}\label{eq:low_frequencies}
C e^{C_0/h} \| u \|_{H^{-N}},
\end{equation}
where the constant $C_0$ only depends on $\Lambda$ and $R$. 

Let us assume now that $|\alpha|>R$. Then we estimate the imaginary part of the phase  $\Phi_{TS}(\alpha,\beta)$, assuming that $\alpha \in \Omega'$ and $\beta \in T^* \mathcal{M}$ are at distance at most $\eta$. Write $\alpha = e^{H_G^{\omega_I}} \gamma$ for some $\gamma \in \Omega$. Using that the imaginary part of $\Phi_{TS}(\gamma,\beta)$ is non-negative (see \cite[Lemma 2.13]{BJ20}), we find
\begin{equation*}
\begin{split}
\Im \Phi_{TS}(\alpha,\beta) & \geq \Im\p{\Phi_{TS}(e^{H_G^{\omega_I}}\gamma,\beta) - \Phi_{TS}(\gamma,\beta)} \\
       & \geq \mathrm{d}_\alpha \Im \Phi_{TS}(\gamma,\beta) \cdot H_G^{\omega_I}(\gamma) + \mathcal{O}\p{\frac{(\log \jap{\va{\alpha}})^2}{\jap{\va{\alpha}}}} \\
       & \geq \mathrm{d}_\alpha \Im \Phi_{TS}(\alpha,\alpha) \cdot H_G^{\omega_I}(\alpha) + \mathcal{O}\p{\frac{(\log \jap{\va{\alpha}})^2}{\jap{\va{\alpha}}} + \eta \log\jap{\va{\alpha}}} \\
       & \geq \Im ( \theta(H_G^{\omega_I})(\alpha)) + \mathcal{O}\p{\frac{(\log \jap{\va{\alpha}})^2}{\jap{\va{\alpha}}} + \eta \log\jap{\va{\alpha}}}.
\end{split}
\end{equation*}
Here, we used the fact that $G$ has logarithmic order (which implies that both $H_G^{\omega_I}$ and the distance between $\alpha$ and $\gamma$ are $\mathcal{O}(\log\langle|\alpha|\rangle/\langle|\alpha|\rangle)$) and that $\mathrm{d}_\alpha \Phi_{TS}(\alpha,\alpha) = \theta$, where $\theta$ is the canonical $1$-form on the cotangent bundle of the complexification $\widetilde{\mathcal{M}}$ of $\mathcal{M}$ (see \cite[(2.73)]{BJ20}). We recall the explicit formula \cite[(2.9)]{BJ20} for the weight $H$:
\begin{equation*}
H = \int_0^1 \p{e^{(\tau - 1)H_G^{\omega_I}}}^* (G - \Im \theta(H_G^{\omega_I})) \mathrm{d}\tau,
\end{equation*}
in order to see that 
\begin{equation*}
\begin{split}
H(\alpha) & = G(\alpha) - \Im ( \theta(H_G^{\omega_I})(\alpha))  + \mathcal{O}\p{\frac{(\log \jap{\va{\alpha}})^2}{\jap{\va{\alpha}}}} \\
          & = G(\gamma) - \Im ( \theta(H_G^{\omega_I})(\alpha))  + \mathcal{O}\p{\frac{(\log \jap{\va{\alpha}})^2}{\jap{\va{\alpha}}}}.
\end{split}
\end{equation*}
It follows, using assumption \eqref{assumption_lemma:distribution_space}, that
\begin{equation*}
\begin{split}
\Im \Phi_{TS}(\alpha,\beta) & \geq G(\gamma) - H(\alpha) + \mathcal{O}\p{\frac{(\log \jap{\va{\alpha}})^2}{\jap{\va{\alpha}}} + \eta \log\jap{\va{\alpha}}} \\
    & \geq \delta \log \jap{\va{\alpha}} - H(\alpha) + \mathcal{O}\p{\frac{(\log \jap{\va{\alpha}})^2}{\jap{\va{\alpha}}} + \eta \log\jap{\va{\alpha}}} \\
    & \geq \frac{\delta}{2} \log \jap{\va{\alpha}} - H(\alpha),
\end{split}
\end{equation*}
provided $\eta$ is small enough and $R$ is large enough. We insert this estimate in our formula
\begin{align*}
|I(\alpha)| &\leq  e^{\frac{H(\alpha)}{h}} \int_{d_{KN}(\alpha,\beta)<\eta} e^{-\frac{\delta \log \langle|\alpha|\rangle}{2h}}  |e(\alpha,\beta)| |Tu(\beta)|\mathrm{d}\beta, \\
			&\leq  e^{\frac{H(\alpha)}{h}} \langle|\alpha|\rangle^{-\frac{\delta}{2h}} h^{-n}\int_{d_{KN}(\alpha,\beta)<\eta}  |Tu(\beta)|\mathrm{d}\beta,\\
			&\leq  e^{\frac{H(\alpha)}{h}} \langle|\alpha|\rangle^{\frac{n}{2}+N-\frac{\delta}{2h}} h^{-n-N}\|u\|_{H^{-N}}. 
\end{align*}
Since, $T_\Lambda u(\alpha)$ coincides with $I(\alpha)$ up to exponentially decaying terms, we deduce that 
\[
\int_{\alpha\in \Omega',\ |\alpha|>R } |T_\Lambda u(\alpha)|^2  e^{-2\frac{H(\alpha)}{h}} \langle|\alpha|\rangle^{2k} \mathrm{d}\alpha \leq C_N e^{-1/Ch} \|u\|_{H^{-N}}^2.
\]
The result then follows since for low frequencies $e^{-H(\alpha)/h} \va{T_\Lambda u (\alpha)}$ is bounded by \eqref{eq:low_frequencies}.
\end{proof}

The point of working with a logarithmic weight $G$ is that $C^\infty$ regularity is enough to prove that a distribution belongs to $\mathcal{H}_\Lambda^k$, as expressed in the following lemma. However, notice that the value of the integral in \eqref{eq:apriori_estimate} could \emph{a priori} blows up when $h$ tends to $0$. The main point in the proof of Proposition \ref{prop:radial_estimate} below is to prove that it does not happen.

\begin{lemma}\label{lemma:space_cinfini}
Let $G$ be a symbol of logarithmic order on $(T^* \mathcal{M})_\epsilon$. Assume that $G$ is small enough as a symbol of order $1$. Let $\Omega$ be a closed conic subset of $T^* \mathcal{M}$. Let $\Omega' = e^{H_G^{\omega_I}} \Omega$. Let $u \in \mathcal{D}'(\mathcal{M})$ be such that $\WF(u) \cap \Omega = \emptyset$, where $\WF(u)$ denotes the $C^\infty$ wave front set of $u$. Then, for every $k \in \R$, and $h>0$ small enough
\begin{equation}\label{eq:apriori_estimate}
\int_{\Omega'} \va{T_\Lambda u}^2 e^{- \frac{2H}{h}} \jap{\va{\alpha}}^{2k} \mathrm{d}\alpha < + \infty.
\end{equation}
\end{lemma}

\begin{proof}
Since the wave front set of $u$ does not intersect $\Omega$, it follows for instance from \cite[Theorem 4.8]{WunschZworski}  that for every $N > 0$, there is a constant $C_N > 0$ such that for $\alpha \in \Omega$, we have
\begin{equation}\label{eq:wfcinfini}
\va{Tu(\alpha)} \leq C_N h^N \jap{\va{\alpha}}^{-N}.
\end{equation}
Then, we can estimate the imaginary part of $\Phi_{TS}(\alpha,\beta)$ as in the proof of Lemma \ref{lemma:distribution_space}. Since we do not make an ellipticity assumption on $G$ anymore, we only find that $\Im \Phi_{TS}(\alpha,\beta) +H(\alpha)$ is a $\mathcal{O}(\log \jap{\va{\alpha}})$ for $\alpha \in \Lambda$ and $\beta \in T^* \mathcal{M}$ close to each other. Hence, using \eqref{eq:formule_avec_noyau} to estimate $e^{- \frac{H}{h}} \va{T_\Lambda u}$, we find that there is a constant $C_0$, that does not depend on $N$ (while $C_N$ does), such that for $\alpha \in \Omega'$ large, we have
\begin{equation}\label{eq:wffcinfini}
e^{- \frac{H(\alpha)}{h}}\va{T_\Lambda u(\alpha)} \leq C_N \jap{\va{\alpha}}^{\frac{C_0}{h} - N}.
\end{equation}
The bound \eqref{eq:apriori_estimate} follows by taking $N$ large enough (depending on $h$).
\end{proof}

If we consider the analytic wave front set instead of the $C^\infty$ wave front set, then the dependence on $h$ of the bound in Lemma \ref{lemma:space_cinfini} becomes explicit.

\begin{lemma}\label{lemma:space_analytic}
Let $G$ be a symbol of logarithmic order on $(T^* \mathcal{M})_\epsilon$. Assume that $G$ is small enough as a symbol of order $1$. Let $\Omega$ be a closed conic subset of $T^* \mathcal{M}$. Let $\Omega' = e^{H_G^{\omega_I}} \Omega$. Let $u \in \mathcal{D}'(\mathcal{M})$ be such that $\WF_a(u) \cap \Omega = \emptyset$. Then for every $k \in \R$ there is a constant $C > 0$ such that, for $h$ small enough, we have
\begin{equation*}
\int_{\Omega'} \va{T_\Lambda u}^2 e^{- \frac{2H}{h}} \jap{\va{\alpha}}^{2k} \mathrm{d}\alpha \leq C e^{\frac{C}{h}}.
\end{equation*}
\end{lemma}

\begin{proof}
The proof is the same as for Lemma \ref{lemma:space_cinfini} with \eqref{eq:wfcinfini} replaced by \eqref{eq:strong_wave_front_set}. Then \eqref{eq:wffcinfini} becomes
\begin{equation*}
e^{- \frac{H(\alpha)}{h}}\va{Tu(\alpha)} \leq C \exp\p{\frac{C_0}{h} \log \jap{\va{\alpha}} - \frac{C_1}{h} \jap{\va{\alpha}}},
\end{equation*}
for some $C,C_0,C_1 > 0$ and $\alpha \in \Omega'$ large enough. The result follows (small frequencies give rise at most to an $\mathcal{O}(e^{C/h})$).
\end{proof}

Finally, we give a result that allows to get regularity estimates from estimates in the space $\mathcal{H}_\Lambda^0$.

\begin{lemma}\label{lemma:hlambda_to_wfs}
Let $G$ be a symbol of logarithmic order. Assume that $G$ is small enough as a symbol of order $1$. Let $\Omega$ be a conic subset of $T^* \mathcal{M} \setminus \{0\}$. Assume that there is $\delta > 0$ such that for every $\alpha \in \Omega$ large enough we have
\begin{equation*}
G(\alpha) \leq - \delta \log \jap{\alpha}.
\end{equation*}
Let $u \in \mathcal{D}'(\mathcal{M})$ not depend on $h$. Assume that $u \in \mathcal{H}_\Lambda^0$ and that there is a constant $C > 0$ such that, for $h$ small enough, we have
\begin{equation*}
\n{u}_{\mathcal{H}_\Lambda^0} \leq C \exp\p{\frac{C}{h}},
\end{equation*}
then the analytic wave front set of $u$ does not intersect $\Omega$.
\end{lemma}

\begin{proof}
We will work as in the proofs of Lemmas \ref{lemma:distribution_space} and \ref{lemma:space_cinfini}, except that we go from $\Lambda$ to $T^* \mathcal{M}$ instead of going from $T^* \mathcal{M}$ to $\Lambda$. The formula \eqref{eq:formule_avec_noyau} becomes for $\alpha\in\Omega$
\begin{equation}\label{eq:formule_avec_noyau_2}
T u(\alpha) = \int_{\Lambda} K_{TS}(\alpha,\beta) T_\Lambda u(\beta) \mathrm{d}\beta
\end{equation}
This formula is deduced from \eqref{eq:formule_avec_noyau} by a contour shift, and the convergence is ensured by the decay of $K_{TS}$ away from the diagonal, and the estimate we assumed on $T_\Lambda u$. Then, we work as in the proof of Lemma \ref{lemma:distribution_space} to estimate the imaginary part of the phase, taking into account the sign shifts (notice in particular that $\mathrm{d}_\beta \Phi_{TS}(\beta,\beta) = - \theta$, where $\theta$ is the canonical $1$-form on $(T^* \mathcal{M})_\epsilon$), we find that for $\alpha \in \Omega$ and $\beta \in \Lambda$ close to each other and large enough we have
\begin{equation*}
\Im \Phi_{TS}(\alpha,\beta)- H(\beta) \geq \frac{\delta}{2} \log \jap{\va{\alpha}}.
\end{equation*}
Since $\Omega$ is closed, this estimate actually holds on a conic neighbourhod of $\Omega$ (up to taking $\delta$ slightly smaller). Neglecting off-diagonal terms as always, it follows from \eqref{eq:formule_avec_noyau_2} and the fact that $e^{- \frac{H}{h}} T u$ is square integrable on $\Lambda$ that
\begin{equation}\label{eq:presque}
\va{Tu(\alpha)} \leq C e^{\frac{C}{h} - \frac{\delta}{4} \frac{\log \jap{\va{\alpha}}}{h}}
\end{equation}
for $\alpha$ large enough in a conic neighbourhood of $\Omega$. Let then $\Omega_0$ be a compact part of $T^* \mathcal{M}$ such that $\Omega \subseteq \bigcup_{t > 0} t \Omega_0$ and the elements of $\Omega_0$ are large enough so that \eqref{eq:presque} gives that $\va{Tu(\alpha)}$ is a $\mathcal{O}(e^{-1/Ch})$ for $\alpha \in \Omega_0$. According to Proposition \ref{prop:equivalence_wave_front_set}, the set $\Omega_0$ does not intersect $\WF_a(u)$. The same is true for $\Omega$ since $\WF_a(u)$ is conic.
\end{proof}

\subsection{Analytic radial estimates}\label{ssec:estimates}

We leave now the generality from \S \ref{ssec:FBItranform}-\ref{ssec:analytic_wave_front_set} to come back to the geometric context of Proposition \ref{analyticityofPi0}. Our analysis will be based on the framework from \cite{DG16}, and we will try to use notations that are coherent with this reference.

We work on the unit tangent bundle $S M_e$ of the extension $(M_e,g_e)$ of $(M,g)$. In order to be consistent with \cite{DG16}, we will write 
$\overline{\mathcal{U}}$ instead of $S M_e$. Notice that $\overline{\mathcal{U}}$ is a real-analytic manifold with boundary. We denote the interior of  $\overline{\mathcal{U}}$ by $\mathcal{U}$ and its boundary by $\partial \mathcal{U}$. We embed $\overline{\mathcal{U}}$ into a closed real-analytic manifold $\mathcal{M}$. To see that it is possible, notice that the Sasaki metric on $\overline{\mathcal{U}}$ is analytic, hence there is an analytic inward pointing vector field defined on the neighbourhood of $\partial \mathcal{U}$ given by the gradient of the distance to $\partial \mathcal{U}$. Thus, $\partial \mathcal{U}$ admits a collar neighbourhood that identifies with $\partial \mathcal{U} \times \left[0,\epsilon \right[$ \emph{via} real-analytic coordinates. We can use this collar neighbourhood to put a real-analytic structure on the double manifold $\mathcal{M}$ of $\overline{\mathcal{U}}$.

As in \S \ref{ssec:FBItranform}-\ref{ssec:analytic_wave_front_set}, we endow $\mathcal{M}$ with a real-analytic Riemannian metric (\emph{a priori} unrelated to the metric $g_e$ on $M_e$) and define an analytic FBI transform $T$ on $\mathcal{M}$.
For notational simplicity, we drop in this section the index $g_e$ and denote by $X$ the generator of the geodesic flow on $SM_e$, 
by $\varphi_t$ its flow, by $\Gamma_\pm$ the incoming/outgoing trapped sets, etc.
The vector field $X$ is real-analytic on $\overline{\mathcal{U}}$ and satisfies the assumptions (A1)-(A5) from \cite{DG16} (see \cite[\S 5.2]{DG16}).  As in \cite{DG16}, we extend $X$ to a vector field on $\mathcal{M}$ such that $\mathcal{U}$ and $\overline{\mathcal{U}}$ are convex for the flow of $X$. In order to have the intermediate results from \cite{DG16} available, we assume that the extension of $X$ is precisely the one constructed in \cite{DG16}. There is \emph{a priori} no reason for $X$ to be analytic away from $\overline{\mathcal{U}}$, however, it is clear from the proof of \cite[Lemma 1.1]{DG16} that we can assume that $X$ is analytic on a neighbourhood of $\overline{\mathcal{U}}$. Thanks to Theorem \ref{theorem:local_multiplication_formula}, this is enough in order to apply the methods from \cite{BJ20}.

We will still denote by $(\varphi_t)_{t \in \mathbb{R}}$ the flow generated by $X$ and by $(\Phi_t)_{t \in \mathbb{R}}$ the lift of $\varphi_t$ to $T^* \mathcal{M}$:
\begin{equation*}
\Phi_t(\alpha) = \p{\varphi_t (\alpha_x), (\mathrm{d}\varphi_t(\alpha_x)^{-1})^{\top}\alpha_\xi}, \quad \alpha = (\alpha_x,\alpha_\xi) \in T^* \mathcal{M}.
\end{equation*}
We denote by $\widetilde{X}$ the generator of the flow $(\Phi_t)_{t \in \mathbb{R}}$. We also recall that the  $\Gamma_\pm=\Gamma_\pm^g$ and $\mc{K}=\mc{K}^g$ are defined by \eqref{def:Gammapm} and \eqref{def:trappedset}. Finally, we denote by $p$ the principal symbol of $X$, which is also the symbol of the semi-classical differential operator $hX$:
\begin{equation*}
p(\alpha) = i \alpha_\xi(X(\alpha_x)) \quad \textup{ for } \alpha \in T^* \mathcal{M}.
\end{equation*}
Notice that this formula also defines a holomorphic extension for $p$.

This section is dedicated to the proof of the following estimate which is crucial in the proof of Proposition \ref{analyticityofPi0}. 

\begin{proposition}\label{prop:radial_estimate}
Let $u$ be a distribution on $\mathcal{U}$. Assume that $Xu$ is compactly supported in $\mathcal{U}$, that the analytic wave front set $\WF_a(Xu)$ of $Xu$ does not intersect $E_+^*$, that the $C^\infty$ wave front set $\WF(u)$ of $u$ does not intersect $E_+^*$ and that $u_{|\partial_+ \mathcal{U}} = 0$. Then the analytic wave front set $\WF_a(u)$ of $u$ does not intersect $E_+^*$.
\end{proposition}

Notice that replacing $X$ by $-X$, we also get:

\begin{proposition}\label{prop:radial_estimate_bis}
Let $u$ be a distribution on $\mathcal{U}$. Assume that $Xu$ is compactly supported in $\mathcal{U}$, that the analytic wave front set $\WF_a(Xu)$ of $Xu$ does not intersect $E_-^*$, that the $C^\infty$ wave front set $\WF(u)$ of $u$ does not intersect $E_-^*$ and that $u_{|\partial_- \mathcal{U}} = 0$. Then the analytic wave front set $\WF_a(u)$ of $u$ does not intersect $E_-^*$.
\end{proposition}

\begin{remark}
We refer to Propositions \ref{prop:radial_estimate} and \ref{prop:radial_estimate_bis} as radial estimates since their proof relies on the source/sink structure of $E_-^*$ and $E_+^*$ for the Hamiltonian flow $(\Phi_t)_{t \in \R}$. The $C^\infty$ analogues of Propositions \ref{prop:radial_estimate} and \ref{prop:radial_estimate_bis} in the case of Anosov flows follow from the radial estimates \cite[Theorems E.52 and E.54]{DyZw}, and the case of open systems is dealt with in \cite{DG16}. 

The analytic radial estimates that we prove here for open systems also apply to Anosov flows. The same proof using the tools from \cite{BJ20} would also give a Gevrey version of Propositions \ref{prop:radial_estimate} and \ref{prop:radial_estimate_bis}. It is also likely that one could deal with more general operators with source/sink structure.

See \cite{GaZw} for a similar statement near some smooth radial submanifolds.
\end{remark}

\subsubsection{Construction of an escape function}

The main tool in the proof of Proposition \ref{prop:radial_estimate} is the construction of an escape function.

\begin{lemma}\label{lemma:escape_function}
Let $\Omega_0$ be a conic neighbourhood of $E_+^*$ in $T^* \mathcal{M}$ and $K_0$ be a compact subset of $\mathcal{U}$. There is a symbol $G$ of logarithmic order on $(T^* \mathcal{M})_\epsilon$, a conic neighbourhood $\Omega \subseteq \Omega_0$ of $E_+^*$ in $T^* \mathcal{M}$ and a constant $C > 0$ such that for $\alpha \in T^* \mathcal{M}$ with $\jap{\alpha} \geq C$ and $\alpha_x \in K_0$, we have:
\begin{enumerate}[label=(\roman*)]
\item if $\alpha \notin \Omega$, then $G(\alpha) \geq C^{-1} \log \jap{\alpha}$;
\item if $\alpha \in \Omega$, then $-H_{\Re p}^{\omega_I} G(\alpha) \leq -C^{-1}$;
\item if $\alpha \in E_+^*$, then $G(\alpha) \leq - C^{-1} \log \jap{\alpha}$.
\end{enumerate}
\end{lemma}

The proof of Lemma \ref{lemma:escape_function} will require the following fact, that follows from \cite[Lemma 1.11]{DG16}.

\begin{lemma}\label{lemma:everywhere}
Let $\kappa : T^* \mathcal{M} \setminus \set{0} \to S^* \mathcal{M}$ be the canonical projection on the cosphere bundle $S^* \mathcal{M}$ of $\mathcal{M}$. Let $U_+$ be a small enough neighbourhood of $\kappa(E_+^*)$ in $S^* \mathcal{M}$. Then, there is a function $\epsilon : \R_+ \to \R_+$ that tends to $0$ in $+ \infty$ and constants $C,\tilde{\gamma} > 0$ such that, for all $\alpha \in T^* \mathcal{M} \setminus \set{0}$ and $t \geq 0$ such that $\kappa(\alpha) \in U_+,\alpha_x \in \overline{\mathcal{U}}$ and $\varphi_t( \alpha_x) \in \overline{\mathcal{U}}$, we have
\begin{equation*}
\begin{split}
d_{S^*\mc{M}}(\kappa(\Phi_t (\alpha)),\kappa(E_+^*)) \leq \epsilon(t) \textup{ and }
\va{\Phi_t (\alpha)} \geq C^{-1} e^{\tilde{\gamma}t} \va{\alpha}.
\end{split}
\end{equation*}
\end{lemma}

\begin{proof} The main difference between Lemma \ref{lemma:everywhere} and \cite[Lemma 1.11]{DG16} is that we do not require here that $\alpha$ belongs to the kernel of the principal symbol $p$ of $X$. This restriction can be removed by noticing that the Liouville $1$-form $\alpha_L$ is invariant by the flow $(\phi_t)_{t \in \mathbb{R}}$ and $p(\alpha_L)=i$. Hence, we can write $\alpha = \beta + c \alpha_L$ with $\beta$ in the kernel of $p$ and then apply \cite[Lemma 1.11]{DG16} to $\beta$. \end{proof}

With Lemma \ref{lemma:everywhere} at our disposal, we can adapt the proof of \cite[Lemma 1.12]{DG16} to prove the following. Here, we let $\widetilde{X}$ act on $S^* \mathcal{M}$ by pushing it forward by the action of $\kappa$.

\begin{lemma}\label{lemma:escape_first_step}
Let $U_+$ be a small enough neighbourhood of $\kappa(E_+^*)$ in $S^* \mathcal{M}$. Then, there is $m_+ \in C^\infty(S^* \mathcal{M})$ with the following properties:
\begin{enumerate}
\item $m_+ \equiv 1$ on a neighbourhood of $\kappa(E_+^*)$ and $0 \leq m_+ \leq 1$ everywhere;
\item $\textup{ supp } m_+ \cap \kappa(T^*_{\mathcal{U}} \mathcal{M}) \subseteq U_+$;
\item if $\alpha \in T^*_{\mathcal{U}} \mathcal{M}$, then $\widetilde{X} m_+(\kappa(\alpha)) \geq 0$;
\item there is $\delta > 0$ such that if $\alpha \in T^*_{\mathcal{U}} \mathcal{M} \setminus \set{0}$ and $\frac{1}{4} \leq m_+(\kappa(\alpha)) \leq \frac{3}{4}$ then $\widetilde{X} m_+(\kappa(\alpha)) \geq \delta$.
\end{enumerate}
\end{lemma}

\begin{proof}
First, \cite[Lemma 1.4]{DG16} says that if $W$ is a neighbourhood of the trapped set $\mc{K}^g$, $\exists T>0$ large enough so that 
\begin{equation}\label{approcheK} 
x\in \mc{U},\varphi_{2T}(x)\in \mc{U} \Longrightarrow \varphi_T(x)\in W
\end{equation}
and \cite[Lemma 1.3]{DG16} says that if $x\in \Gamma_+$, then $d_{\mc{M}}(\varphi_t(x),\mc{K}^g)\to 0$ as $t\to -\infty$.
Let us choose a $C^\infty$ function $m_0: S^* \mathcal{M}\to \left[0,1\right]$, supported in $U_+ \cap \kappa(T^*_{\mathcal{U}} \mathcal{M})$ and such that $m_0 \equiv 1$ on a neighbourhood of $\kappa(E_u^*) = \kappa(E_+^* \cap T_{\mc{K}^g}^* \mathcal{M})$. Denote by $W$ the projection of this neighbourhood to $\mathcal{M}$. Then, for $T >0$ large, define the function
\begin{equation*}
m_+ : \alpha \mapsto \frac{1}{T} \int_{-2T}^{-T} m_0(\kappa(\Phi_t(\alpha))) \mathrm{d}t.
\end{equation*}
Let us show that $m_+$ has the required properties when $T$ is chosen large enough. The proof is the same as  \cite[Lemma 1.12]{DG16}, we summarize the argument for the convenience of the reader:
\begin{enumerate}[label=\emph{(\arabic*)}]
	\item It follows from a continuity argument and the fact that for $W$ as above and all $\alpha\in \kappa(E_+^*)$, we have $\Phi_t(\alpha)\in \kappa(E_+^* \cap T_{W}^* \mathcal{M})$ for all $t\in [-2T,-T]$ if $T>0$ is large enough.
	\item If $\alpha\in \supp(m_+)$ and $\alpha_x\in\mathcal{U}$, then $\exists t\in [-T,-2T]$ such that $\Phi_t(\alpha)\in \supp(m_0)$, and $\alpha_x,\varphi_t(\alpha_x)\in \mc{U}$ while $\kappa(\Phi_t(\alpha))\in U_+$, then Lemma \ref{lemma:everywhere} implies that $\alpha\in U_+$, assuming $T$ is large enough.
	\item For $\alpha_x\in \mc{U}$, one has $T \widetilde{X}m_+(\alpha)=m_0(\kappa(\Phi_{-T}(\alpha))-m_0(\kappa(\Phi_{-2T}(\alpha)))$. Assume that $m_0(\kappa(\Phi_{-T}(\alpha))<1$ and $m_0(\kappa(\Phi_{-2T}(\alpha)))>0$, then $\varphi_{-2T}(\alpha_x)\in \mc{U}$ and $\varphi_{-T}(\alpha_x)\in W$ if $T$ is large enough (by \eqref{approcheK}), but $d_{S^*\mc{M}}(\kappa(\Phi_{-T}(\alpha)),\kappa(E_u^*))>\eps$  for some $\eps>0$ depending only on $m_0$. Since also  $\kappa(\Phi_{-2T}(\alpha))\in U_+$, Lemma \ref{lemma:everywhere} gives a contradiction with $d_{S^*\mc{M}}(\kappa(\Phi_{-T}(\alpha)),\kappa(E_u^*))>\eps$ if $T>0$ is chosen large enough. Thus $\widetilde{X}m_+(\alpha)\geq 0$.
	\item Let $\alpha \in T^*_{\mathcal{U}} \mathcal{M} \setminus \set{0}$ be such that $\frac{1}{4} \leq m_+(\kappa(\alpha)) \leq \frac{3}{4}$. 
Since $m_+(\kappa(\alpha)) \leq \frac{3}{4}$, then the proportion of $t$'s in $[-2T,-T]$ such that $m_0(\kappa(\Phi_{t}(\alpha)))= 1$ is at most $\frac{3}{4}$. Hence, there is a $t \in \left[ - \frac{7}{4}T,-T\right]$ such that $m_0(\kappa(\Phi_{t}(\alpha)))< 1$. Thus, if $T > 0$ is large enough, the same argument as for \emph{(3)} based on Lemma \ref{lemma:everywhere} shows that $m_0(\kappa(\Phi_{-2T}(\alpha))) = 0$. Similarly, we get $m_0(\kappa(\Phi_{-T}(\alpha)))=1$, so that $\widetilde{X} m_+(\kappa(\alpha)) = \frac{1}{T}$.\qedhere
\end{enumerate}
\end{proof}

\begin{remark}
Before we give the proof of Lemma \ref{lemma:escape_function}, let us link $H^{\omega_I}_{\Re p}$ to $\Phi_t$. From equation (2.3) in \cite{BJ20}, 
we see that for a smooth function $f$, we have in local coordinates $\tilde{x} = x+ i y, \tilde{\xi} = \xi + i \eta$
\begin{equation*}
H_f^{\omega_I} = \sum_{j = 1}^n \partial_{\eta_j} f \partial_{x_j} - \partial_{x_j} f \partial_{\eta_j} + \partial_{\xi_j} f \partial_{y_j} - \partial_{y_j} f \partial_{\xi_j}.
\end{equation*}
If $f$ vanishes on the reals, we see that $H_f^{\omega_I}$ is tangent to the reals. If additionally $f= \Re q$, $q$ being holomorphic and pure imaginary on the reals, we can use the Cauchy--Riemann equations to find
\begin{align*}
H_f^{\omega_I} 	&= \sum_{j = 1}^n \partial_{\eta_j} f \partial_{x_j} - \partial_{y_j} f \partial_{\xi_j} + \mathcal{O}(|y|+|\eta|) \\ & = \sum_{j = 1}^n \Re(\partial_{\eta_j} q)\partial_{x_j} -\Re( \partial_{y_j} q )\partial_{\xi_j} + \mathcal{O}(|y|+|\eta|)\\
				&= -\sum_{j = 1}^n \partial_{\xi_j} \Im(q)\partial_{x_j} - \partial_{x_j} \Im(q)\partial_{\xi_j} + \mathcal{O}(|y|+|\eta|)
\end{align*}
In the last line, we recognize the usual Hamilton vector field of $-\Im q$ on the reals. In particular, 
\begin{equation}\label{eq:Rep-Xtilde}
- H^{\omega_I}_{\Re p} = \tilde{X} \quad \textup{ on } T^* \mathcal{M}.
\end{equation}
\end{remark}

\begin{proof}[Proof of Lemma \ref{lemma:escape_function}]
Let $r$ be a non-negative symbol of logarithmic order on $T^* \mathcal{M}$ such that $r(\alpha) = \log \jap{\alpha}$ when $\alpha$ is large. Let 
\begin{equation*}
G_0(\alpha) = \int_{-T}^0 r(\Phi_t( \alpha)) \mathrm{d}t,
\end{equation*}
for $T$ large enough. Thanks to Lemma \ref{lemma:everywhere}, we see that there is a neighbourhood $U_+$ of $\kappa(E_+^*)$ in $S^* \mathcal{M}$ and a constant $C > 0$ such that 
\begin{equation}\label{XG0}
\forall \alpha \in \kappa^{-1}\p{U_+} \cap T_{\mathcal{U}}^ * \mathcal{M},\quad C\geq  \widetilde{X} G_0(\alpha) \geq C^{-1}.  
\end{equation} 
Then, up to making $U_+$ smaller, we can assume that $U_+ \subseteq \kappa(\Omega_0)$ and apply Lemma \ref{lemma:escape_first_step}. Let $m_+$ be the resulting function. We let $\tilde{m}$ be a symbol of order $0$ on $T^* \mathcal{M}$ that coincides with $1 - 2 m_+ \circ \kappa$ outside of a bounded set. Define then
\begin{equation*}
G = \tilde{m} G_0.
\end{equation*}
For $\alpha \in T^*_{\mathcal{U}} \mathcal{M}$ large enough, using \eqref{eq:Rep-Xtilde}, we have
\begin{equation*}
- H_{\Re p}^{\omega_I} G(\alpha) = -2 \widetilde{X}m_+(\kappa(\alpha)) G_0(\alpha) + (1 - 2 m_+ \circ \kappa(\alpha)) \widetilde{X} G_0(\alpha).
\end{equation*}
Let $\Omega =  \kappa^{-1} (\{m_+ \geq \frac{1}{4}\})\subset \kappa^{-1}(U_+)$, and notice that $G\geq \frac{1}{2} \log \cjg \alpha\cjd$, if $|\alpha|$ is large enough and not in $\Omega$; this gives point \emph{(i)}.

If $\alpha \in E_+^*$ then $m_+(\kappa(\alpha)) = 1$, so that $G(\alpha) = - G_0(\alpha)$, and (iii) follows. Now, if $\alpha \in \Omega$ is large, either $m_+(\kappa(\alpha)) \geq \frac{3}{4}$, in which case we have by 
point (3) of Lemma \ref{lemma:escape_first_step} and \eqref{XG0}
\begin{equation*}
- H_{\Re p}^{\omega_I} G(\alpha) \leq - \frac{1}{2} C^{-1},
\end{equation*}
or $m_+(\kappa(\alpha)) \leq \frac{3}{4}$, in which case $\widetilde{X}m_+(\kappa(\alpha)) \geq \delta$ by point (4) of Lemma \ref{lemma:escape_first_step}, and thus
\begin{equation*}
-H_{\Re p}^{\omega_I} G(\alpha) \leq - 2 \delta G_0(\alpha)+ C
\end{equation*}
and (ii) follows since $G_0\geq \eps \log \cjg\alpha\cjd$ for large $|\alpha|$.
\end{proof}

\subsubsection{Proof of Proposition \ref{prop:radial_estimate}}
Set $f = Xu$. We start by reducing to the case of $u$ compactly supported in $\mathcal{U}$. If $x \in \Gamma_+ \cap \partial \mathcal{U}$, then $u$ vanishes on a neighbourhood of $x$, since $Xu$ does, $u_{|\partial_+ \mathcal{U}} = 0$ and $X$ is transversal to the outgoing boundary. By compactness of $\Gamma_+ \cap \partial \mathcal{U}$, we find a compact subset $K_1$ of $\mathcal{U}$ such that $\textup{supp } f \subseteq K_1$ and $\overline{\mathcal{U} \setminus K_1}$ does not intersect $\Gamma_+ \cap \textup{ supp } u$. Let $\chi \in C^\infty(\mathcal{M})$ be supported in $\mathcal{U}$ and such that $\chi \equiv 1$ on a neighbourhood of $K_1$. Then $\chi u$ is a distribution on $\mathcal{M}$, supported in $\mathcal{U}$ and
\begin{equation*}
X(\chi u) = \chi f + (X \chi) u = f + (X \chi) u.
\end{equation*}
Notice that $X \chi$ is supported in $\mathcal{U} \setminus K_1$, so that $(X \chi) u$ is supported away from $\Gamma_+$. In particular the analytic wave front set of $(X \chi) u$, and hence of $f + (X \chi) u$ does not intersect $E_+^*$. Since $\chi$ is $C^\infty$, and since $\WF(u)\cap E_+^*=\emptyset$ by the assumptions on $u$, then  $\WF(\chi u)\cap E_+^*=\emptyset$. Since $\chi \equiv 1$ on a neighbourhood of $\Gamma_+ \cap \textup{ supp }u$, we have $\WF_a(u) \cap E_+^* = \WF_a(\chi u) \cap E_+^*$. Consequently, we may replace $u$ by $\chi u$ and $f$ by $f + (X \chi)u$ and assume that $u$ is compactly supported in $\mathcal{U}$. We let $K_0 \subseteq \mathcal{U}$ denote a compact neighbourhood of $\textup{supp } u$ and $\textup{supp }f$.

By assumption, there is a conic neighbourhood $\Omega_0$ of $E_+^*$ in $T^* \mathcal{M}$ such that $\WF(u) \cap \Omega_0 = \emptyset$ and $\WF_a(f) \cap \Omega_0 = \emptyset$. Let $G$ and $\Omega$ be as in Lemma \ref{lemma:escape_function}. Multiplying $G$ by a smooth function equal to $1$ 
outside $\{|\alpha|>R\}$ and supported in  $\{|\alpha|>R/2\}$, we can assume that $G$ is small enough as a symbol of order $1$. Let $\Lambda \coloneqq \Lambda_G$ and decompose $\Lambda$ as
\begin{equation*}
\Lambda = \Lambda_1 \cup \Lambda_2 \cup \Lambda_3,
\end{equation*}
where $\Lambda_1,\Lambda_2$ and $\Lambda_3$ are respectively the images of $T^*_{\mathcal{M} \setminus K_0} \mathcal{M}$, $T^*_{K_0} \mathcal{M} \setminus \Omega$ and $T^*_{K_0} \mathcal{M} \cap \Omega$ by $e^{H_G^{\omega_I}}$. According to Lemma \ref{lemma:pseudo_local_space} and since $u$ is supported in the interior of $K_0$, for every $N>0$ large enough and all $k \in \R$, there is $C = C_{k} > 0$ such that
\begin{equation}\label{eq:Lambda1}
\int_{\Lambda_1} \va{T_\Lambda u}^2 \jap{\va{\alpha}}^{2k} e^{- \frac{2H}{h}} \mathrm{d}\alpha \leq C \exp\p{- \frac{C}{h}} \|u\|^2_{H^{-N}}.
\end{equation}
According to the properties of $G$ (Lemma \ref{lemma:escape_function}), $G$ is larger than $\delta \log\langle|\alpha|\rangle$ for $\alpha$ large enough in $T^*_{K_0} \mathcal{M} \setminus \Omega$, so we can apply Lemma \ref{lemma:distribution_space} and obtain
\begin{equation}\label{eq:Lambda2}
\int_{\Lambda_2} \va{T_\Lambda u}^2 \jap{\va{\alpha}}^{2k} e^{- \frac{2H}{h}} \mathrm{d}\alpha \leq C \exp\p{\frac{C}{h}} \|u\|^2_{H^{-N}}.
\end{equation}
Now, since the $C^\infty$ wave front set of $u$ does not intersect $\Omega_0$, we see from Lemma \ref{lemma:space_cinfini}, that for $h > 0$ small enough and every $k \in \mathbb{R}$, we have
\begin{equation*}
\int_{\Lambda_3} \va{T_\Lambda u}^2 e^{- \frac{2H}{h}} \jap{\va{\alpha}}^{2k} \mathrm{d}\alpha < + \infty.
\end{equation*}
Hence, we have $u \in \mathcal{H}_\Lambda^\infty$, and we may apply the multiplication formula of Proposition \ref{theorem:local_multiplication_formula} with $P = h X$ to find that
\begin{equation*}
\Re\p{\int_{\Lambda} T_\Lambda u \overline{T_\Lambda P u} e^{- \frac{2H}{h}} \mathrm{d}\alpha} \leq \int_{\Lambda} \Re q \va{T_\Lambda u}^2 e^{- \frac{2H}{h}} \mathrm{d}\alpha + C h \int_{\Lambda} \va{T_\Lambda u}^2 e^{- \frac{2H}{h}} \mathrm{d}\alpha  + C\|u\|_{H^{-N}}^2e^{-\frac{C}{h}},
\end{equation*}
where $q$ is the restriction to $\Lambda$ of the principal symbol $p$ of $P$. Since $q$ is a symbol of order $1$, we may apply \eqref{eq:Lambda1} and \eqref{eq:Lambda2} with $k = \frac{1}{2}$ to find that
\begin{equation*}
\int_{\Lambda} \Re q \va{T_\Lambda u}^2 e^{- \frac{2H}{h}} \mathrm{d}\alpha \leq \int_{\Lambda_3}  \Re q \va{T_\Lambda u}^2 e^{- \frac{2H}{h}} \mathrm{d}\alpha + Ch \int_{\Lambda_3} \va{T_\Lambda u}^2 e^{- \frac{2H}{h}} \mathrm{d}\alpha + C e^{\frac{C}{h}}\|u\|^2_{H^{-N}}.
\end{equation*}
Now, if $\alpha \in \Lambda_3$, we have $\alpha = e^{H_G^{\omega_I}} \beta$ with $\beta \in T^*_{K_0} \mathcal{M} \cap \Omega$. Hence, recalling that $p$ denotes the holomorphic extension of the symbol of $P$, we have
\begin{equation*}
\begin{split}
\Re q(\alpha) & = \Re p(\alpha) = \Re p(\beta) + H_G^{\omega_I} \Re p(\beta) + \mathcal{O}\p{\frac{\log \jap{\va{\alpha}}^2}{\jap{\va{\alpha}}}} \\
    & = - H_{\Re p}^{\omega_I} G + \mathcal{O}\p{\frac{\log \jap{\va{\alpha}}^2}{\jap{\va{\alpha}}}}  \\
    & \leq - C^{-1},
\end{split}
\end{equation*}
where the last inequality holds for $\alpha$ large enough and we used the fact that ${\rm Re}(p(\beta))=0$ for $\beta\in T^*\mc{M}$ a real covector. Since small frequencies always produce a term that grows at most exponentially fast when $h$ tends to $0$, we have
\begin{equation*}
\int_{\Lambda_3}  \Re q \va{T_\Lambda u}^2 e^{- \frac{2H}{h}} \mathrm{d}\alpha \leq - C^{-1} \int_{\Lambda_3}  \va{T_\Lambda u}^2 e^{- \frac{2H}{h}} \mathrm{d}\alpha + C e^{\frac{C}{h}}\|u\|^2_{H^{-N}}.
\end{equation*}
It follows that
\begin{equation*}
\Re\p{\int_{\Lambda} T_\Lambda u \overline{T_\Lambda P u} e^{- \frac{2H}{h}} \mathrm{d}\alpha} \leq (- C^{-1} + Ch) \int_{\Lambda_3} \va{T_\Lambda u}^2 e^{- \frac{2H}{h}} \mathrm{d}\alpha +C e^{\frac{C}{h}}\|u\|^2_{H^{-N}},
\end{equation*}
so that for $h$ small enough, we have
\begin{equation*}
\begin{split}
\int_{\Lambda_3} \va{T_\Lambda u}^2 e^{- \frac{2H}{h}} \mathrm{d}\alpha  & \leq- C \Re\p{\int_{\Lambda} T_\Lambda u \overline{T_\Lambda P u} e^{- \frac{2H}{h}} \mathrm{d}\alpha} + C e^{\frac{C}{h}}\|u\|^2_{H^{-N}} \\
    & \leq C \delta \int_{\Lambda} \va{T_\Lambda u}^2 e^{- \frac{2H}{h}} \mathrm{d}\alpha + C \delta^{-1} \int_{\Lambda} \va{T_\Lambda P u}^2 e^{- \frac{2H}{h}} \mathrm{d}\alpha + C e^{\frac{C}{h}}\|u\|^2_{H^{-N}} \\
    & \leq C \delta \int_{\Lambda_3} \va{T_\Lambda u}^2 e^{- \frac{2H}{h}} \mathrm{d}\alpha + C \delta^{-1} \int_{\Lambda_3} \va{T_\Lambda P u}^2 e^{- \frac{2H}{h}} \mathrm{d}\alpha + C e^{\frac{C}{h}}\|u\|^2_{H^{-N}}.
\end{split}
\end{equation*}
Here $\delta$ is any small positive number, and we noticed that \eqref{eq:Lambda1} and \eqref{eq:Lambda2} still holds when $u$ is replaced by $P u$, for the same reasons. Since the analytic wave front set of $Pu = h f$ does not intersect $\Omega_0$, we deduce from Lemma \ref{lemma:space_analytic} 
that there is a constant $C(u)>0$ depending on $u$ such that
\begin{equation*}
 \int_{\Lambda_3} \va{T_\Lambda P u}^2 e^{- \frac{2H}{h}} \mathrm{d}\alpha \leq C(u) e^{\frac{C(u)}{h}}.
\end{equation*}
So that by taking $\delta$ small enough and using \eqref{eq:Lambda1} and \eqref{eq:Lambda2} again, we see that
\begin{equation*}
\n{u}_{\mathcal{H}_\Lambda^0}^2 \leq C(u) e^{\frac{C(u)}{h}}.
\end{equation*}
Hence, since $G$ is elliptic on $E_+^*$, it follows from Lemma \ref{lemma:hlambda_to_wfs} that $\WF_a(u) \cap E_+^* = \emptyset$.
\qed

\subsection{Proof of Proposition \ref{analyticityofPi0}}\label{ssec:analytic_pseudor}

We now have at our disposal all the microlocal estimates required from the proof of Proposition \ref{analyticityofPi0}, which will follow from the following result.

\begin{lemma}\label{lemma:regularizing}
Let $v \in \mathcal{D}'(\mathring{M}_e)$ be supported in the interior of $M_e$. If $x$ does not belong to the support of $v$ then $\Pi^{g_e}_0v$ is analytic on a neighbourhood of $x$.
\end{lemma}

In addition to microlocal estimates, the proof of Lemma \ref{lemma:regularizing} will require to understand the dynamics of the geodesic flow lifted to $T^* (S M_e) \subseteq T^* \mathcal{M}$. To do so, introduce the horizontal bundle, defined for $x \in S M_e$ by
\begin{equation*}
H^*(x) = \{ \xi \in T^*_x (SM_e) \, |\, \xi_{| \ker \mathrm{d} \pi_0(x)} = 0 \} \quad \textup{ for } x \in SM_e.
\end{equation*}
Here, we recall that $\pi_0$ is the canonical projection $S M_e \to M_e$. The importance of the bundle $H^*$ is due to the following two results.

\begin{lemma}\label{lemma:going_up}
Let $u \in \mathcal{D}'(\mathring{M}_e)$ be compactly supported, then the analytic wave front set of $\pi_0^* u$ is contained in $H^*$. The same holds true for $\pi_2^*u$ if $u \in \mc{D}'(\mathring{M}_e; S^2 T^*\mathring{M}_e)$.
\end{lemma}

\begin{lemma}\label{lemma:going_down}
Let $u \in \mathcal{D}'(\mathcal{M})$ be supported in the interior of $SM_e$. Let $x$ be a point in the interior of $M$ and assume that for every $y \in \pi_0^{-1}(\{x\})$ the horizontal direction $H^*(y)$ does not intersect $\WF_a(u)$. Then $\pi_{0*}u$ is analytic on a neighbourhood of $x$. The same holds true for ${\pi_2}_*u$ if $u \in \mc{D}'(\mathring{M}_e; S^2 T^*\mathring{M}_e)$.
\end{lemma}

Lemma \ref{lemma:going_up} is a consequence of \cite[Theorem 8.5.1]{HoeI} and the fact that $\pi_0$ is analytic. 
Lemma \ref{lemma:going_down} follows from \cite[\S 3.b)]{Laubin} (see also \cite[Theorem 8.5.4]{HoeI}), and the fact that the projection of the analytic wave front set to the physical space is the analytic singular support \cite[Theorem 6.3]{SjAsterisque}. The arguments apply as well for $\pi_2^*$ and ${\pi_2}_*$. 

Most of the geometric information needed for the proof of Lemma \ref{lemma:regularizing} is contained in the following lemma, 
which is a consequence of the hyperbolicity of the trapped set and the absence of conjugate points.
\begin{lemma}\label{lemma:geometric}
For  $x \in \Gamma_\pm^{g_e}$ one has $H^*(x) \cap E_\pm^*(x) = \{0\}$. Moreover , if $\alpha=(\alpha_x,\alpha_{\xi})\in T^*(SM_e)$ is such that $\alpha_\xi(X(x))=0$,
$\alpha_{\xi} \in H^*(x)$ and  $(d\varphi^{g_e}_t(x)^{-1})^\top\alpha_{\xi}\in H^*(\varphi^{g_e}_t(x))$, then $\alpha_\xi=0$.
\end{lemma}
\begin{proof}
The proof is contained in the proof of \cite[Proposition 5.1]{Guillarmou17}:
the fact that $H^*(x) \cap E_\pm^*(x) = \{0\}$ follows from the hyperbolicity of $\mc{K}^g$ and a result of Klingenberg \cite[Proposition p.6]{Kl}, and the second statement is equivalent to the fact that $(\pi_0(\varphi^{g_e}_s(x)))_{s\in [0,t]}$ is free of conjugate points.
\end{proof}

We are now ready to prove Lemma \ref{lemma:regularizing}.

\begin{proof}[Proof of Lemma \ref{lemma:regularizing}]
Let $v \in \mathcal{D}'(\mathring{M}_e)$ be supported in the interior of $M_e$. Write $f = \pi_0^*v$ and $u = R_{g_e}^*f$. Then $\Pi_0^{g_e}v = \pi_{0*} u$, see \eqref{eq:symetrique}. Notice that we have $X_{g_e} u = - f$ with $u= 0$ near $\partial_+ SM_e$, and by \eqref{WFR*} and the composition rules for wave-front sets we have $\WF(u)\cap E_+^*=\emptyset$.

Let $x$ be a point that does not belong to the support of $v$. According to Lemma \ref{lemma:going_down}, we only need to prove that the analytic wave front set of $u$ does not intersect $H^*(y)$ for all $y \in \pi_0^{-1}(\{x\})$.

Let $y$ be a point of $\pi_0^{-1}(\{x\})$ and $\alpha \in H^*(y)$. From Lemma \ref{lemma:going_up}, we know that the analytic wave front set of $f$ is contained in $H^* \cap \pi_0^{-1}(\textup{supp } v)$. In particular, $\alpha \notin \WF_a(f)$ (since $\pi_0(y) = x \notin \textup{ supp } v$). If $p(\alpha) \neq 0$, with $p$ the principal symbol of $X$, we know that $\alpha \notin \WF_a(u)$ by ellipticity of $X$ (see \cite[Theorem 8.6.1]{HoeI} or \cite[Theorem 6.4]{SjAsterisque}).

Let us then consider the case $p(\alpha) = 0$. Thanks to Lemma \ref{lemma:geometric}, we know that the orbit of $\alpha$ for the lift of the geodesic flow to $T^*(SM_e)$ never intersects the analytic wave front set of $f$. Hence, by propagation of singularity (see for instance \cite[Theorem 2.9.1]{HitrikSjostrand}), we only need to prove that this orbit leaves the analytic wave front set of $u$. If $y \notin \Gamma_-^{g_e}$, then $\varphi_t^{g_e}(y)$ must intersect $\partial_+ \mathcal{U}$, and $u$ vanishes near $\partial_+ \mathcal{U}$. If $y \in \Gamma_-^{g_e}$, then $\alpha \notin E_-^*$ due to Lemma \ref{lemma:geometric}, and thus $\Phi_t(\alpha)$ converges to $E_+^*$ due to \cite[Lemma 1.10]{DG16}. Hence, we only need to prove that the analytic wave front set of $u$ does not intersect $E_+^*$ (since the analytic wave front set is closed). This fact follows from Proposition \ref{prop:radial_estimate} since $Xu =- f$, the analytic wave front set of $f$ does not intersect $E_+^*$ by Lemmas \ref{lemma:going_up} and \ref{lemma:geometric}, the $C^\infty$ wave front set of $u$ does not intersect $E_+^*$ and $u=0$ near $\pl_+\mc{U}$.
\end{proof}

From Lemma \ref{lemma:regularizing} and some functional analysis, we deduce the following: 

\begin{corol}\label{corollary:analytic}
Let $K_0$ and $K_1$ be disjoint compact subsets of $\mathring{M}_e$. Let $N \in \R$. Then there is a constant $C > 0$ and a complex neighbourhood $V$ of $K_1$ such that if $u \in H^{-N}(\mathring{M}_e)$ is supported in $K_0$, then $\Pi_0^{g_e} u$ has a holomorphic extension to $V$ bounded by $C \n{u}_{H^{-N}}$.
\end{corol}

\begin{proof}
Let $H^{-N}(K_0)$ denote the subspace of $H^{-N}(\mathring{M}_e)$ made of the elements that are supported in $K_0$. Let $(V_n)_{n \geq 0}$ be a decreasing family of complex neighbourhoods of $K_1$ such that $K_1 = \bigcap_{n \geq 0} V_n$. For every $m,n \geq 0$, introduce
\begin{equation*}
F_{m,n} = \{u \in H^{-N}(K_0) : \Pi_0^{g_e} u \textup{ has a holomorphic extension to } V_n \textup{ bounded by } m\}.
\end{equation*}
Then, Lemma \ref{lemma:regularizing} implies that
\begin{equation*}
H^{-N}(K_0) = \bigcup_{m,n \geq 0} F_{m,n}.
\end{equation*}
Let us prove that the $F_{m,n}$'s are closed in $H^{-N}(K_0)$. Fix $m$ and $n$ and let $(u_p)_{p \geq 0}$ be a sequence of elements of $F_{m,n}$ that converges to $u \in H^{-N}(K_0)$. Set $v_p = \Pi_0^{g_e} u_p$ and let denote by $\tilde{v}_p$ the holomorphic extension of $v_p$ to $V_n$. By Montel's Theorem, we may assume, up to extracting a subsequence, that $\tilde{v}_p$ converges uniformly on all compact subsets of $V_n$ to a holomorphic function $\tilde{v}$ bounded by $m$. Since $\Pi_0^{g_e}$ is a ($C^\infty$) pseudo-differential operator, we know that $v_p$ converges to $\Pi_0^{g_e} u$ as a distribution, so that the restrictions of $\tilde{v}$ and $\Pi_0^{g_e} u$ to $V_n \cap M_e$ coincides, proving that $u \in F_{m,n}$.

Since the $F_{m,n}$'s are closed, it follows from Baire's Theorem that one of them has non-empty interior. Hence, for some $n \geq 0$, the linear subspace
\begin{equation*}
\bigcup_{m \geq 0} F_{m,n}
\end{equation*}
has non-empty interior, and is thus equal to $H^{-N}(K_0)$. The result follows then from the uniform boundedness theorem applied to the family of linear forms $u \mapsto \Pi_0^{g_e} u(x)$ for $x \in V_n$. 
\end{proof}

\begin{proof}[Proof of Proposition \ref{analyticityofPi0}]
Apply $\Pi_0^{g_e}$ to a Dirac mass to deduce from Corollary \ref{corollary:analytic} that, away from the diagonal, the kernel of $\Pi_0^{g_e}$ is analytic in $x$ uniformly in $y$. The result follows then from Lemma \ref{lemma:deux_variables}, the fact that $\Pi_0^{g_e}$ is symmetric and the fact that the projection of the analytic wave front set to the physical space is the analytic singular support.
\end{proof}

\begin{proof}[Proof of Proposition \ref{proposition:analytic_pseudo}]
Since we already know that the kernel of $\Pi_0^{g_e}$ is analytic away from the diagonal, we only need to understand the kernel of $\Pi_0^{g_e}$ near the diagonal. Let us choose $x_0\in \mathring{M}_e$ and $\eps>0$ very 
small so that the ball $B_{\eps}:=\{x\in M_e\, | \,  d_{g_e}(x,x_0)<\eps\}$ satisfies $d_{g_e}(B_{\eps},\pl M_e)>5\eps$, the radius of injectivity of $(M_e,g_e)$ is larger than $10\eps$ and the boundary of $B_{\eps}$ is strictly convex. 
In particular $B_{\eps}$ is a simple manifold, in the terminology of the introduction. Let us consider $A: C_c^\infty(B_{\eps})\to C^\infty(B_{\eps})$ defined by $Au:=(\Pi_0^{g_e}u)|_{B_{\eps}}$. Thanks to Proposition \ref{analyticityofPi0}, it is enough to prove that $A$ is analytic pseudo-differential operator.
By the proof of \cite[Proposition 5.7]{Guillarmou17}, we can write
\[  R_{g_e}^* =  e^{3\eps X_{g_e}}R_{g_e}^* + \int_0^{3\eps}e^{tX_{g_e}}dt, \quad \Pi_0^{g_e}= {\pi_0}_*e^{3\eps X_{g_e}}R_{g_e}^*\pi_0^*+{\pi_0}_*
\int_0^{3\eps}e^{tX_{g_e}}dt\pi_{0}^*\] 
Let us write $A=A_1+A_2$ where $A_1u=({\pi_0}_*\int_0^{3\eps}e^{tX_{g_e}}dt\pi_{0}^*u)|_{B_{\eps}}$ for all $u\in C_c^\infty(B_\eps)$. 
By \cite[Section 5.1]{Guillarmou17}, $A_1=(I_0^{B_{\eps}})^*I_0^{B_{\eps}}$ is the normal operator for the X-ray transform $I_0^{B_{\eps}}$ 
on the simple manifold $(B_{\eps},g_e)$. 
In \cite[Proposition 3.2]{Stefanov-Uhlmann-JAMS2005}, Stefanov and Uhlmann proved that a similar operator is a real-analytic pseudo-differential operator. The same analysis applies here if we replace the formula \cite[(2.1)]{Stefanov-Uhlmann-JAMS2005} by the formula for the kernel of the normal operator given in \cite[Lemma 3.1]{Pestov-Uhlmann}. Hence, we find that $A_1$ is an elliptic analytic pseudo-differential on $B_\eps$.  Next we consider $A_2$. We can use the same argument as above, in the proof of Proposition \ref{analyticityofPi0} to see that $A_2$ has an analytic kernel. The only difference is that, adapting the proof of Lemma \ref{lemma:regularizing}, we work with $u = e^{3\eps X_{g_e}}R_{g_e}^*f$ for $f=\pi_0^*v$ and $v$ a distribution supported in $B_{\eps}$,  and $A_2v = (\pi_{0*} u)|_{B_\eps}$. We only need to use that $e^{3\eps X_{g_e}}: f\mapsto f\circ \varphi_{3\eps}^{g_e}$, with $\varphi_{3\eps}^{g_e}$ an analytic diffeomorphism, and that $10\eps$ is smaller than the radius of injectivity of $(M_e,g_e)$ to deduce that $A_2v$ is analytic in $B_\eps$ (the singular support of $\pi_{0*} u$ is at least at distance $\eps$ from $B_\eps$).
\end{proof}

\begin{proof}[Proof of Proposition \ref{injectiviteI2}]
By using again \cite[Proposition 3.2]{Stefanov-Uhlmann-JAMS2005} (which contains the analysis for symmetric $2$-tensors on simple manifolds) to deal with the operator 
\begin{equation}\label{localdescription}
 u\in C_c^\infty(B_\eps ;S^2T^*B_{\eps}) \mapsto \Big({\pi_2}_*\int_0^{3\eps}e^{tX_{g_e}}dt\pi_{2}^*u\Big)|_{B_{\eps}}\in C^\infty(B_\eps ;S^2T^*B_{\eps}),
 \end{equation}
the same exact argument as in the proof of Proposition \ref{proposition:analytic_pseudo} shows that the normal operator $\Pi_2^{g_e} = (I_2^{g_e})^*I_2^{g_e}$ is an analytic pseudo-differential operator with principal symbol being the same as for simple manifolds. The only difference consists in changing $\pi_0^*,{\pi_{0}}_*$ by $\pi_2^*,{\pi_2}_*$, which is harmless for the wave-front set analysis as these operators share the same exact (analytic) microlocal properties. 
Once we know this fact and that $\Pi_2^{g}=(I_2^g)^*I_2^g$, the proof of solenoidal injectivity is an adaptation of the proof of \cite[Theorem 1]{Stefanov-Uhlmann08}. We summarize the argument 
of that proof for the convenience of the reader and since our geometric assumption is not the same:\\
1) Any $L^2(M)$ symmetric $2$-tensor can be decomposed as $f=f^s+D_gw$ with $D_g^*f^s=0$ and $w\in H^1(M;T^*M)$ with $w|_{\pl M}=0$, 
with $D_g:C^\infty(M;T^*M)\to C^{\infty}(M;S^2T^*M)$ the symmetrized Levi-Civita covariant derivative, $D_g^*$ is its formal adjoint (called divergence).  The tensor $f^s$ is called the solenoidal projection of $f$ and one has $I_2^gf=I_2^gf^s$ since $I_2^gD_gw=I^gX_g\pi_1^*w=0$. 
This decomposition can be done by setting $w=(\Delta_{D_g})^{-1}D_g^*f$ where $\Delta_{D_g}=D^*_gD_g$ is the symmetric Laplacian with Dirichlet condition.
See \cite[\S 2.3]{Stefanov-Uhlmann08} or \cite[Theorem 3.3.2]{Sharfutdinov-Book} for details.\\
2) The operator $\Pi_2^g$ is elliptic on solenoidal tensors in the sense that its principal symbol $p(x,\xi)$ 
is injective on $\{f\in S^2T_x^*M\,|\,  f(\xi^\sharp,\cdot)=0\}$ if $\xi^\sharp\in T_xM$ is the dual vector to $\xi$ by $g_x$. This follows from the local description
\eqref{localdescription} of $\Pi_2^g$ up to smoothing operators. Then by analytic ellipticity of $\Pi_2^g$ (since $\Pi_2^g$ is an analytic pseudodifferential operator as mentioned above), we find that if $I_2^gf^s=0$, then $f^s$ is analytic in the interior of $M$. Then, working as in \cite[Lemma 6]{Stefanov-Uhlmann08} (see also \cite[Proposition 3.4]{Stefanov-Uhlmann-JAMS2005}), we find $f^s\in \mc{A}(M)$, that is $f^s$ is analytic up to the boundary. Note that to get the property near the boundary, we use $\Pi_2^{g_e}$ and the fact that $g_e$ is analytic on $M_e$.
\\

3) According to \cite[Lemma 4]{Stefanov-Uhlmann08}, if $I_2^gf^s=0$, there exists an analytic one form $w_0\in \mc{A}(U;T^*M)$ vanishing on $\partial M$, so that $h:=f^s-D_gw_0$ satisfies $h(\psi_*\pl_r,\cdot)=0$ and $h$ vanishes to infinite order at $\pl M$, if $\psi:[0,\eps]\times \pl M\to U\subset M$ is the normal form of Lemma \ref{normalform}. Since $h$ is analytic in $U$, this implies that $h=0$ and thus $f^s=D_gw_0$ near $\pl M$.\\

4) The analytic $1$-form $w_0$ defined in $U$ admits an analytic extension to $M$ and $f^s=D_gw_0$ on $M$, which implies that $w_0=0$ since $f^s$ is the solenoidal projection of $f$. To prove the existence of the extension of $w_0$, we proceed as in the proof of 
\cite[Theorem 1]{Stefanov-Uhlmann08}.  
For $x\in M\setminus U$, take
$x_0\in \pl M$ and $\gamma$ a geodesic between $x_0$ and $x$. Let $S\subset U$ be a local codimension $1$ submanifold
orthogonal to $\gamma$, we can then use geodesic normal coordinates $(t,y)$ associated to $S$, 
covering a neighborhood $V_{\gamma}$ of $\gamma$. We can then solve the equation 
$(f^s-D_gw_1)(\pl_t,\cdot)=0$ in $V_\gamma$ with initial condition $w_1|_S=w_0|_S$. As shown in the proof of \cite[Lemma 4]{Stefanov-Uhlmann08}, this is an ODE which has a unique analytic solution, and this solution is equal to $w_0$ in $U$; this gives an extension of $w_1$ in $V_\gamma$, in particular near $x$. Moreover $f^s=D_gw_1$ in $V_\gamma$ by unique continuation from $U\cap V_\gamma$.
We need to check that this extension is independent of the choice of $x_0$ and $\gamma$. Let $x_0'$ and $\gamma'$ another pair as above and consider 
the analytic continuation of $w_1$ from $V_\gamma\cap V_{\gamma'}$ to $V_{\gamma'}$ by the process describes above, but starting from a submanifold 
$S'$ passing through $x$ orthogonal to $\gamma'$. We need to show that $w_1=w_0$ in $U\cap V_{\gamma'}$.
First, we observe that by the process above, $w_1$ extends analytically in a neighborhood of all broken geodesic $\gamma''$ with endpoints $x_0,x_0'$, that is obtained by a homotopy of broken geodesics from
$\gamma\cup\gamma'$,  and the value of $w_1$ near $x_0'$ is independent of such curve by unique continuation. 
Let us then take the unique (non broken) geodesic 
$\gamma_0:[0,\ell_{\gamma_0}]$ in this homotopy class (it is unique by \cite[Lemma 2.2]{Guillarmou-Mazzucchelli-18} since $g$ has no conjugate points). Let us consider the finite dimensional manifold $G_{k}(x_0,x_0',\gamma_0)$ of piecewise geodesics $c:[0,1]\to M$   in the homotopy class of $\gamma_0$ with $(c(0),c(1))=(x_0,x_0')$, with $k$ pieces  and with each geodesic segment of size less than the injectivity radius ($k$ is chosen large enough). 
A standard argument using the gradient flow of the energy functional $E(c)=\int_0^1 \|\dot c(t)\|^2_g dt$ on
$G_{k}(x_0,x_0',\gamma_0)\cap E^{-1}([0,E(\gamma\cup \gamma')])$ shows that there is a homotopy 
from $\gamma\cup\gamma'$ to $\gamma_0$ in $G_{k}(x_0,x_0',\gamma_0)$ (see for example \cite[Theorem 16.3]{Milnor-63} and its adaptation to the case with strictly convex boundary as in the proof of \cite[Lemma 2.2]{Guillarmou-Mazzucchelli-18}).  
Let us define $v_0:=\dot{\gamma}_0(0)\in \pl_-SM$, $v_0':=\dot{\gamma}_0(\ell(\gamma_0))$ and $S_g(x_0,v_0)=(x_0',v_0')$, then we have 
\[I_2^gf^s(x_0',v_0')=0=w_1(x_0')v_0'-w_1(x_0)v_0=w_1(x_0')v_0'-w_0(x_0)v_0=w_1(x_0')v_0'\]
since $w_0$ vanishes on $\pl M$. 
Since similarly one has $I_2^gf^s(x_0',v)=0=w_1(x_0')v$ for all $(x_0',v)\in \pl_+SM$ near $(x_0',v_0')$, we obtain $w_1(x_0')=0$. Moving now $x_0'$ slightly,
the same argument shows that $w_1=0$ on  
$\pl M$ near $x_0'$. Since $f^s=D_gw_0=D_gw_1$ near $x_0'$ and $w_1=w_0=0$ on $\pl M$ near $x_0'$, we deduce that $w_1=w_0$ in a neighborhood of $x_0'$ in $M$ (the solution of this PDE being unique locally). We have thus proved the result.
\end{proof}

\appendix

\section{Technical results on the analytic wave front set}\label{appendix:wave_front_set}

For the convenience of the reader, we recall here the definition of the analytic wave front set of a distribution \cite[Definition 6.1]{SjAsterisque}. As the notion is invariant by analytic change of coordinates, we give the definition on an open subset of $\mathbb{R}^n$.

\begin{definition}\label{definition:wave_front_set}
Let $U$ be an open susbet of $\mathbb{R}^n$ and $u \in \mathcal{D}'(U)$. Let $(x_0,\xi_0) \in T^* \mathbb{R}^n \setminus \{0\} \simeq \mathbb{R}^n \times (\mathbb{R}^n \setminus \{0\})$. Let: 
\begin{itemize}
\item $\varphi(x,\alpha)$ be an analytic function on a neighbourhood of $(x_0,x_0,\xi_0)$ such that $\varphi(x,\alpha) = 0$ and $\mathrm{d}_x \varphi(x,\alpha) = \alpha_\xi$ when $x = \alpha_x$, and $\Im \varphi(x,\alpha_x) \geq C^{-1} |x - \alpha_x|^2$ when $x$ and $\alpha$ are real.
\item $a_\lambda(x,\alpha)$ be an analytic function on a neighbourhood of $(x_0,x_0,\xi_0)$ depending on a parameter $\lambda \geq \lambda_0$. We assume that $a_\lambda(x,\alpha)$ is bounded and bounded away from zero, uniformly in $\lambda \geq \lambda_0$ on a fixed complex neighbourhood of $(x_0,x_0,\xi_0)$.
\item $\chi \in \mathcal{C}_c^\infty(U)$ supported near $x_0$ and identically equal to $1$ on a neighbourhood of $x_0$.
\end{itemize}
Then $(x_0,\xi_0)$ does not belong to the analytic wave front set of $u$ if and only if 
\begin{equation*}
\int_{\mathbb{R}^n} e^{i \lambda \varphi(x,\alpha)} a_\lambda(x,\alpha) \chi(x) u(x) \mathrm{d}x
\end{equation*}
decays exponentially fast when $\lambda$ tends to $+ \infty$, uniformly for $\alpha$ in a neighbourhood of $(x_0,\xi_0)$.
\end{definition}

It follows from \cite[Proposition 6.2]{SjAsterisque} that this definition does not depend on the choice of $\varphi,a_\lambda$ or $\chi$.

\begin{proof}[Proof of Proposition \ref{prop:equivalence_wave_front_set}]
The equivalence between (i) and (ii) is an immediate consequence of \cite[Proposition 6.2]{SjAsterisque}. It is clear that (iii) implies (ii), so let us prove that (ii) implies (iii).

We consider thus $\Omega_0$ a bounded neighbourhood of $\alpha_0\in T^\ast\mathcal{M}\setminus\{0\}$ where $Tu$ is exponentially small in $h$. Then we pick $\Omega_1$ another neighbourhood of $\alpha_0$, contained and relatively compact in $\Omega_0$, and let $\Omega = \cup_{t>0} t \Omega_1$. For $\alpha\in\Omega$, we write $\alpha = t \alpha'$ with $\alpha'\in\Omega_1$ and $t \in \R_+^*$.

Assuming $t>1$ and letting $\tilde{h} = h/t$, we will need to distinguish the FBI transform with semi-classical paramater $h$ and $\tilde{h}$, so that we will write $T_h$, $T_{\tilde{h}}$ instead of just $T$ and $S_h$, $S_{\tilde{h}}$ instead of $S$ to make it clear. Let us then use the inversion formula to write $T_h u = T_h S_{\tilde{h}} T_{\tilde{h}} u$ for $h$ and $\tilde{h}$ small enough. The action of the operator $S_{\tilde{h}}$ on functions that grows at most polynomially is defined by duality. Hence, for $\alpha \in T^* \mathcal{M}$, we have
\begin{equation}\label{eq:changement_de_h}
T_h S_{\tilde{h}} T_{\tilde{h}}u (\alpha) = \int_{T^* \mathcal{M}} K_{h,\tilde{h}}(\alpha,\beta) T_{\tilde{h}}u(\beta) \mathrm{d}\beta,
\end{equation}
where the kernel $K_{h,\tilde{h}}$ is defined by
\begin{equation*}
K_{h,\tilde{h}}(\alpha,\beta) = \int_{\mathcal{M}} K_{T_h}(\alpha,y) K_{S_{\tilde{h}}}(y,\beta) \mathrm{d}y.
\end{equation*}
The integral in \eqref{eq:changement_de_h} is convergent as $T_{\tilde{h}} u$ grows at most polynomially  and $K_{h,\tilde{h}}(\alpha,\beta)$ decays exponentially fast in $\beta$ when $\alpha$ is fixed, since $y \mapsto K_{T_h}(\alpha,y)$ is real-analytic.

Let us start with a very crude estimate on the kernel $K_{h,\tilde{h}}$. It follows from the definition of $K_{T_h}$ and $K_{S_{\tilde{h}}}$ that there is $C$ such that if $\alpha,\beta \in T^* \mathcal{M}$ and $y \in \mathcal{M}$, we have
\begin{equation*}
\va{K_{T_h}(\alpha,y)} \leq C h^{- 3n/4} \jap{\alpha}^{n/4} \textup{ and } \va{K_{S_{\tilde{h}}}(y,\beta)} \leq C \tilde{h}^{- 3n/4}\jap{\beta}^{n/4}.
\end{equation*}
Hence, up to making $C$ larger, we have for $\alpha,\beta \in T^* \mathcal{M}$ that
\begin{equation}\label{eq:estimee_simple}
\va{K_{h,\tilde{h}}(\alpha,\beta)} \leq C (h\tilde{h})^{- 3n/4} \jap{\alpha}^{n/4} \jap{\beta}^{n/4}.
\end{equation}

With these notations, we can split
\begin{equation}\label{eq:splitting-integral}
T_h u(\alpha) = \int_{\Omega_0} K_{h,\tilde{h}}(\alpha,\beta) T_{\tilde{h}} u(\beta) d\beta + \int_{T^\ast\mathcal{M}\setminus\Omega_0} K_{h,\tilde{h}}(\alpha,\beta) T_{\tilde{h}} u(\beta) d\beta. 
\end{equation}
We can use the crude estimate \eqref{eq:estimee_simple} and the assumption (ii) that we made on $u$ to deduce that the first integral is a $\mathcal{O}(e^{-1/C\tilde{h}})=\mathcal{O}( e^{- \langle\alpha\rangle/Ch})$. Since we assumed that $u$ is a distribution, $T u$ grows at most polynomially, and it suffices in order to get (iii) to prove that for $\alpha\in\Omega$ and $\beta\notin \Omega_0$, 
\begin{equation}\label{eq:taille_des_erreurs}
\va{K_{h,\tilde{h}}(\alpha,\beta)} \leq C \exp\p{- C^{-1} \p{\frac{\jap{\alpha}}{h}+ \frac{\jap{\beta}}{\tilde{h}}}},
\end{equation}
which implies that the second integral in \eqref{eq:splitting-integral} is also $\mathcal{O}( e^{-1/C\tilde{h}})$. 

Now, $\alpha\in\Omega$ and $\beta\notin\Omega_0$ implies that $\alpha/h$ and $\beta/\tilde{h}$ are not close to each other, which is the right condition for \eqref{eq:taille_des_erreurs} to hold, as we will now prove. The argument is very close to the proof of \cite[Lemma 2.9]{BJ20}. 

Let us first assume that $\alpha,\beta \in T^* \mathcal{M}$ are such that the distance between $\alpha_x$ and $\beta_x$ is larger than some small $\eta > 0$. Then, we may split the integral defining $K_{h,\tilde{h}}(\alpha,\beta)$ into
\begin{equation}\label{eq:coupe_en_trois}
\begin{split}
& K_{h,\tilde{h}}(\alpha,\beta) \\ & \quad = \p{\int_{D(\alpha_x,\eta/100)} + \int_{D(\beta_x,\eta/100)} + \int_{\mathcal{M} \setminus ( D(\alpha_x,\eta/100) \cup D(\beta_x,\eta/100))}} K_{T_h}(\alpha,y) K_{S_{\tilde{h}}}(y,\beta) \mathrm{d}y.
\end{split}
\end{equation}

The third integral in \eqref{eq:coupe_en_trois} is easily dealt with: by assumptions $K_{T_h}$ and $K_{S_{\tilde{h}}}$ are small there. Let us look at the first integral (the second one is dealt with similarly). It may be rewritten, up to a neglectible term, as
\begin{equation*}
\int_{D(\alpha_x,\eta/100)} e^{i \frac{\Phi_T(\alpha,y)}{h}} a(\alpha,y) K_{S_{\tilde{h}}}(y,\beta) \mathrm{d}y.
\end{equation*}
Here, the phase $\Phi_T(\alpha,y)$ is non-stationary, as $y$ is close to $\alpha_x$ and we have $\mathrm{d}_y \Phi_T(\alpha,\alpha_x) = - \alpha_\xi$, and the function $y \mapsto K_{S_{\tilde{h}}}(y,\beta)$ is real-analytic and has a holomorphic extension bounded by $\mathcal{O}(\exp(- \jap{\va{\beta}}/C \tilde{h}))$ on a complex neighbourhood of $D(\alpha_x,\eta/100)$ whose size is uniform in $\alpha,\beta,h$ and $\tilde{h}$. We can consequently apply the non-stationary phase method \cite[Proposition 1.5]{BJ20} to see that this integral is indeed of the size \eqref{eq:taille_des_erreurs}. Let us mention that we apply here \cite[Proposition 1.4]{BJ20} in the case $s=1$, and that the assumption of positivity of $\Im \Phi_T$ on the boundary of $D(\alpha_x,\eta/100)$ follows from the fourth assumptions we made on $\Phi_T$.

We want now to understand $K_{h,\tilde{h}}(\alpha,\beta)$ when the distance between $\alpha_x$ and $\beta_x$ is less than $2 \eta$, but $\frac{\alpha}{h}$ and $\frac{\beta}{\tilde{h}}$ are away from each other in the Kohn--Nirenberg metric. As $\alpha_x$ and $\beta_x$ are close to each other, we may work in coordinates, and the fact that $\frac{\alpha}{h}$ and $\frac{\beta}{\tilde{h}}$ are away from each other will write $\va{\frac{\alpha_\xi}{h} - \frac{\beta_\xi}{\tilde{h}}} \geq A \eta \p{\frac{\jap{\alpha}}{h} + \frac{\jap{\beta}}{\tilde{h}}}$, for a large constant $A$ to be determined later. We split the integral defining $K_{h,\tilde{h}}(\alpha,\beta)$ in two:
\begin{equation*}
K_{h,\tilde{h}}(\alpha,\beta) = \p{\int_{D(\alpha_x,100 \eta)} + \int_{\mathcal{M} \setminus  D(\alpha_x,100\eta)}} K_{T_h}(\alpha,y) K_{S_{\tilde{h}}}(y,\beta) \mathrm{d}y.
\end{equation*}
The second integral is dealt with as the third integral in \eqref{eq:coupe_en_trois}, it is at most of the size \eqref{eq:taille_des_erreurs}. The first integral may be rewritten as
\begin{equation*}
\int_{D(\alpha_x, 100 \eta)} e^{i \p{\frac{\Phi_T(\alpha,y)}{h} + \frac{\Phi_S(y,\beta)}{\tilde{h}}}} a(\alpha,y) b(\beta,y) \mathrm{d}y,
\end{equation*}
where $a$ and $b$ are analytic symbols. The phase is non-stationary as
\begin{equation*}
\begin{split}
& \mathrm{d}_y\p{\frac{\Phi_T(\alpha,y)}{h} + \frac{\Phi_S(y,\beta)}{\tilde{h}}} \\ & \qquad  = \frac{\mathrm{d}_y \Phi_T(\alpha,\alpha_x)}{h} + \frac{\mathrm{d}_y \Phi_S(\beta_x,\beta)}{\tilde{h}} + \mathcal{O}\p{\p{\frac{\jap{\alpha}}{h} + \frac{\jap{\beta}}{\tilde{h}}} \eta} \\
   & \qquad = \frac{\beta_\xi}{\tilde{h}} - \frac{\alpha_\xi}{h} + \mathcal{O}\p{\p{\frac{\jap{\alpha}}{h} + \frac{\jap{\beta}}{\tilde{h}}} \eta}.
\end{split}
\end{equation*}
Taking $A$ large enough, we get
\begin{equation*}
\va{\mathrm{d}_y\p{\frac{\Phi_T(\alpha,y)}{h} + \frac{\Phi_S(y,\beta)}{\tilde{h}}}} \geq C^{-1} \va{\frac{\jap{\alpha}}{h} + \frac{\jap{\beta}}{\tilde{h}}}.
\end{equation*}
Hence, the non-stationary phase method gives again that \eqref{eq:taille_des_erreurs} holds. Here, we cannot use \cite[Proposition 1.4]{BJ20} directly since $\Phi_T$ is divided by $h$ and $\Phi_S$ by $\tilde{h}$. However, we can apply \cite[Proposition 1.1]{BJ20} and a rescaling argument.

\end{proof}

\begin{proof}[Proof of Lemma \ref{lemma:deux_variables}]
We use here Definition \ref{definition:wave_front_set} of the analytic wave front set. Let $(\alpha,\beta) \in T_x^* \mathcal{M} \times T_y^* \mathcal{M}$ be such that $\alpha_\xi \neq 0$. Choose then $\varphi,\psi$ be real-analytic functions defined respectively on a neighbourhood of $(x,\alpha)$ and $(y,\beta)$ in $\mathcal{M} \times T^* \mathcal{M}$, that satisfy 
\[
\begin{cases}
\varphi(x',\alpha') = 0,\  \mathrm{d}_x \varphi(x',\alpha') = \alpha_\xi' 	& \text{ when }\alpha_x' = x' \\
\psi(y',\beta') = 0,\ \mathrm{d}_y \psi (y',\beta') = \beta_\xi' 			& \text{ when }\beta_x' = y'.
\end{cases}
\]
We also assume that there is $C>0$ so that $\Im \varphi(x',\alpha') \geq C^{-1} d(x',\alpha_x')^2$ and $\Im \psi(y',\beta') \geq C^{-1} d(y' , \beta_x')^2$ for $x',\alpha',y',\beta'$ real, respectively near $x,\alpha,y,\beta$. Let then $\chi$ and $\rho$ be $C^\infty$ functions on $\mathcal{M}$ that are identically equal to $1$, respectively near $x$ and $y$. We want to prove that the integral
\begin{equation}\label{eq:sj_wfs}
\int_{\mathcal{M} \times \mathcal{M}} e^{i \lambda (\varphi(x',\alpha') + \psi(y',\beta'))} \chi(x') \rho(y') u(x',y') \mathrm{d}x' \mathrm{d}y'
\end{equation}
decays exponentially fast when $\lambda$ tends to $+ \infty$ and $\alpha'$ and $\beta'$ are near $\alpha$ and $\beta$. Notice that here we take $a_\lambda = 1$ in the notation from Definition \ref{definition:wave_front_set}, which we can do since the definition does not depend on the choice of $a_\lambda$. Let us write the integral \eqref{eq:sj_wfs} as
\begin{equation*}
\int_{\mathcal{M}} e^{i \lambda \psi(y',\beta')} \rho(y') \p{\int_{\mathcal{M}} e^{i \lambda \varphi(x',\alpha')} \chi(x') u(x',y') \mathrm{d}x'} \mathrm{d}y'.
\end{equation*}
In the inner integral, the $x'$ that are away from $x$ are negligible by positivity of $\Im \varphi$, and the $x'$ near $x$ are dealt with by the non-stationary phase method \cite[Proposition 1.1]{BJ20}, using that $\alpha_\xi \neq 0$. We find that the inner integral decays exponentially fast with $\lambda$ uniformly in $y'$. Since the imaginary part of $\psi(y',\beta')$ is non-negative when $y'$ and $\beta'$ are real, we find that \eqref{eq:sj_wfs} decays exponentially fast with $\lambda$, which ends the proof of the lemma.
\end{proof}

\section{Application to lens rigidity for analytic asymptotically hyperbolic manifolds} 
We notice that Theorem \ref{Th2} implies a \emph{renormalized lens rigidity} result for analytic asymptotically hyperbolic manifolds. These are complete manifolds with curvature tending to $-1$ at infinity, and conformal to compact Riemannian manifolds with a conformal factor blowing up at infinity: $(M,g)$ is the interior of a compact manifold with boundary $\bbar{M}$ and there is a smooth boundary defining function $\rho \in C^\infty(\bbar{M},[0,1])$ such that $(\bbar{M},\rho^2 g)$ is a smooth Riemannian manifold with boundary with $|d\rho|_{\rho^2g}=1$ on $\pl \bbar{M}$. 
Once a conformal representative is chosen in the scattering infinity for $g$ (see for instance the introduction of \cite{Graham-Guillarmou-Stefanov-Uhlmann-19}), the scattering map for the geodesic flow on an asymptotically hyperbolic manifold $(M,g)$ with hyperbolic traped set can be defined as map $S_g:T^*\pl \bbar{M}\to T^*\pl \bbar{M}$, see \cite[Proposition 2.6]{Graham-Guillarmou-Stefanov-Uhlmann-19}, 
and the renormalized length $L_g$ can be defined by taking the constant term in the asymptotic expansion as $\eps\to 0$ of the length $\ell_g(\gamma\cap \{\rho\geq \eps\})$ for complete geodesics $\gamma$ with both endpoints at infinity (see \cite[Definition 4.2]{Graham-Guillarmou-Stefanov-Uhlmann-19}). In the proof of \cite[Theorem 1.4]{Graham-Guillarmou-Stefanov-Uhlmann-19}, it is shown in particular that
two analytic asymptotically hyperbolic manifolds $(M_1,g_1)$ and $(M_2,g_2)$ with $(L_{g_1},S_{g_1})=(L_{g_2},S_{g_2})$ must be isometric near $\pl \bbar{M}$, with an isometry $\psi:U_1\to U_2$ for $U_i$ an open neighborhood of $\pl\bbar{M}_i$. Up to reducing this neighborhood we can choose $U_1$ (and thus $U_2)$ to be strictly convex in $M_i$ since the level sets $\{\rho_i=\eps\}$ are easily seen to be strictly convex if $\eps>0$ is small for the boundary defining functions $\rho_i$ as above (i.e. $|d\rho_i|_{\rho_i^2g}=1$ on $\pl \bbar{M}_i$).
The fact that $S_{g_1}=S_{g_2}$ implies that the compact manifolds 
$(M_1\setminus U_1,g_1)$ and $(M_2\setminus U_2,g_2)$ have the same scattering maps. If we assume in addition that 
$(M_i,g_i)$ have negative curvature, or more generally no conjugate points and a hyperbolic trapped set, then we can use Theorem \ref{Th2} to deduce the following:
\begin{corol}\label{AHcase}
Let $(M_1,g_1)$ and $(M_2,g_2)$ be two real analytic asymptotically hyperbolic manifolds with no conjugate points and with hyperbolic trapped set. If, after a choice of conformal representatives in the conformal infinity of $g_1$ and $g_2$, we have 
$(L_{g_1},S_{g_1})=(L_{g_2},S_{g_2})$, then $(M_1,g_1)$ is isometric to $(M_2,g_2)$.
\end{corol} 
This improves Theorem 1.4 in \cite{Graham-Guillarmou-Stefanov-Uhlmann-19}, which was showing the same result as Corollary \ref{AHcase} under the extra assumption that the relative fundamental group $\pi_1(\bbar{M}_1,\pl\bbar{M}_1)=0$.

\bibliography{References.bib}
\bibliographystyle{amsalpha}

\end{document}